\documentclass[11pt]{article}

\usepackage[margin=1in]{geometry}
\usepackage{upquote}
\usepackage{amsfonts,amsmath,amssymb,amsthm}
\usepackage{booktabs}
\usepackage{hyperref}

\usepackage{graphicx}
\graphicspath{ {images/}}
\usepackage{epstopdf}
\usepackage{caption}
\usepackage{subcaption}
\newtheorem{theorem}{Theorem}[section]
\newtheorem{lemma}{Lemma}[section]
\newtheorem{problem}{Problem}
\theoremstyle{definition}

\numberwithin{equation}{section}

\title{\textbf{A higher order numerical scheme for singularly perturbed parabolic turning point problems exhibiting twin boundary layers}}
\author{Swati Yadav\thanks{Department of Mathematics,
		University of Delhi, Delhi-110007, India.
		({\tt swatiyadav3317@gmail.com}).}, ~Pratima Rai\thanks{ Department of Mathematics, University of Delhi, Delhi-110007, India.
		({\tt prai@maths.du.ac.in}).} ~\thanks{Corresponding author.}}
\date{\today}

\begin{document}
	
	\maketitle
	
	\begin{abstract}
In this article, a parameter-uniform numerical method is presented to solve one-dimensional singularly perturbed parabolic convection-diffusion multiple turning point problems exhibiting two exponential boundary layers. We study the asymptotic behaviour of the solution and its partial derivatives. The problem is discretized using the implicit Euler method for time discretization on a uniform mesh and a hybrid scheme for spatial discretization on a generalized Shishkin mesh. The scheme is shown to be $\varepsilon$-uniformly convergent of order one in time direction and order two in spatial direction upto a logarithmic factor. Numerical experiments are conducted to validate the theoretical results. Comparison is done with upwind scheme on uniform mesh as well as on standard Shishkin mesh to demonstrate the higher order accuracy of the proposed scheme.

		\textbf{MSC classification 2010:} 65M12, 65M50, 65M06, 65M15.
	\end{abstract}
	
\section{Introduction}
Singularly perturbed differential equations are model equations for convection-diffusion processes in various physical phenomena and engineering problems, such as heat and mass transport problem with high Pe\`clet numbers, fluid flow at high Reynolds numbers, the drift-diffusion equation in the modeling of semiconductor devices, financial mathematics, quantum physics, etc. A singularly perturbed equation contains a small parameter $\varepsilon$ multiplied with the highest derivative term. The solution of these problems changes rapidly in a thin region as $\varepsilon$ approaches zero. These layer regions are referred to as boundary layers in fluid mechanics, shock layers in fluid and solid mechanics, edge layers in solid mechanics,  transition points in quantum mechanics, skin layers in electrical applications and Stokes lines and surfaces in mathematics. Classical numerical methods fail to capture the behaviour of the solution of the singularly perturbed problems. In order to overcome this difficulty, various special methods based on the fitted mesh and fitted operator technique have been adopted in literature. For more details one may refer to \cite{farrell,miller,roos}.\newline
In this article, we consider the following class of singularly perturbed parabolic problems on a rectangular domain $Q$ with degenerating convective term 
\begin{align}
\label{1.1}
L_{\varepsilon} u(x,t) &= f(x,t), \quad (x,t) \in Q= \Omega_x \times \Omega_t,\\
\label{1.2}
u(x,t)&=g(x,t), \quad (x,t) \in S, 
\end{align} 
where
\begin{align*}
L_{\varepsilon}u(x,t)= \left \{ \varepsilon \frac{\partial^2}{\partial x^2} + a(x,t) \frac{\partial}{\partial x} - d(x,t) \frac{\partial}{\partial t} - b(x,t) \right \}u(x,t),
\end{align*}
$0 < \varepsilon \ll 1$, $S= \overline{Q} \backslash Q$, $\Omega_x=(-1,1)$ and $\Omega_t= (0,T]$. The coefficients $a(x,t)$, $ b(x,t)$, $ d(x,t) $ and $ f(x,t) $ are sufficiently smooth functions such that 
\begin{align}\label{1.3}
a(x,t) &=- a_{0}(x,t)x^{p} ,\hspace{0.3cm} p \geq 1 ~ \text{is a non-negative odd integer}, ~ \text{for}~ (x,t)\in \overline{Q},\nonumber \\
a_0(x,t) &\geq \alpha_0 > 0,~ \text{for}~ (x,t) \in \overline{Q},\nonumber \\
b(x,t) &\geq \beta > 0 , ~ \text{for}~ (x,t) \in \overline{Q},\nonumber \\
d(x,t) &\geq \gamma \geq 0 ,~\text{for}~ (x,t) \in \overline{Q}.
\end{align}
The set $\overline{Q}_{R}$ and $\overline{Q}_{L}$ are defined as
\begin{align*}
\overline{Q}_{L} = [-1,0] \times [0,T] ~ \text{and} ~ \overline{Q}_{R}=[0,1] \times [0,T].
\end{align*}
Sufficient regularity conditions are imposed on the data of the problem (\ref{1.1})-(\ref{1.3}) which guarantee the smoothness of the solution on the set $\overline{Q}$. The set $S$ is denoted by $S=S_l \cup S_x \cup S_r$, where $S_{x}=[-1,1] \times \{t=0\}$, $S_{l}= \{ x=-1 \} \times (0,T]$, $S_{r}= \{ x=1 \} \times (0,T]$. Also, we define the boundaries corresponding to the domains $Q_L$ and $Q_R$ as $S_L=\overline{Q}_L \backslash Q_L=S_{l} \cup S_{0} \cup (S_x \cap \overline{Q}_L)$, $S_R=\overline{Q}_R \backslash Q_R=S_{r} \cup S_0 \cup (S_x \cap \overline{Q}_R)$, where $S_0= \{x=0 \} \times (0,T]$. The data of the problem (\ref{1.1})-(\ref{1.3}) are assumed to be sufficiently smooth and the compatibility conditions are imposed at the corner points $(-1,0)$ and $(1,0)$ to ensure the desired smoothness of the solution of the problem (\ref{1.1})-(\ref{1.3}) for our analysis. The assumed compatibility conditions will ensure the existence of the unique solution $u(x,t)$ of the problem (\ref{1.1})-(\ref{1.3}) which belongs to $C(\overline{Q}) \cap ( C^{4+\lambda,2+\lambda/2}(\overline{Q}_{L}) \cup C^{4+\lambda,2+\lambda/2}(\overline{Q}_{R})) $, where $\lambda \in (0,1)$.	\newline
The problem (\ref{1.1})-(\ref{1.3}) is an interior turning point problem. For $p > 1$, the problem (\ref{1.1})-(\ref{1.3}) has multiple turning point. Singularly perturbed turning point problems arise in the mathematical modeling of various physical phenomenon. Interior turning point problems are convection-diffusion problems with a dominant convective term and a speed field that changes its sign in catch basin. The characteristic curves of the reduced problem are parallel to the boundaries $S_r$ and $S_l$. The solution of the problem possesses parabolic boundary layers in the neighborhood of $S_l$ and $S_r$.\newline
The singularly perturbed ordinary differential equation with a simple turning point exhibiting twin boundary layers has been extensively studied by many authors \cite{becher,geng2013,kadalbajoo2011, kadalbajoo2010,kadalbajoo2001,kumar2019,natesan2003}. Kadalbajoo et al.~\cite{kadalbajoo2001} constructed a second-order accurate numerical method based on cubic splines with a non-uniform grid to solve a singularly perturbed two point boundary value problem with a turning point exhibiting twin boundary layers. Natesan et al. \cite{natesan2003} proposed a fitted mesh method to approximate the solution of the singularly perturbed one dimensional simple turning point problem exhibiting twin boundary layers. The authors proved that the proposed method is parameter uniform with first-order accuracy. Kadalbajoo et al. \cite{kadalbajoo2010} used B-spline collocation method on a piecewise-uniform Shishkin mesh to approximate the solution of singularly perturbed two-point boundary value problems with interior simple turning point exhibiting twin boundary layers. The method is second-order accurate in the maximum norm. Further, Kadalbajoo et al. \cite{kadalbajoo2011} examined a stiff singularly perturbed boundary value problem with simple interior turning point having two boundary layers. The authors constructed a fitted operator finite difference scheme of Il'in type using cubic splines, collocation and an artificial viscosity.  The reproducing kernel method along with the method of scaling was examined by Geng et al. \cite{geng2013} to solve the singularly perturbed turning point problem with twin boundary layers. Becher et al. \cite{becher} considered the singularly perturbed turning point problem with two boundary layers. The authors applied Richardson extrapolation method on classical finite difference scheme with piecewise uniform Shishkin mesh to improve the order of convergence from $O(N^{-1} \log N)$ to $O(N^{-2} \log ^2N)$. In a recent paper, Devendra \cite{kumar2019} considered the singularly perturbed two point boundary value problem with interior turning point. The author discussed both the cases, one in which the solution exhibit interior layer and the other when twin boundary layers are present in the solution. The numerical method used to approximate the solution of the problem comprised of quintic B-spline collocation method on appropriate piecewise uniform Shishkin mesh. For singularly perturbed two-point boundary value problem with multiple interior turning point one can refer to \cite{vulanovic}. \newline
In this article, we extend the study of singularly perturbed turning point problems exhibiting twin boundary layers to time dependent case. To the best of our knowledge, no paper have analysed the twin boundary layers occurring due to interior multiple turning point for singularly perturbed parabolic convection-diffusion problem. In general, the numerical treatment of turning point problems are more difficult than the non-turning point problems because the convection coefficient vanishes inside the domain. In addition, developing a higher order scheme for such problems is of great importance in the field of numerical analysis. In the process, we develop a scheme which comprises of implicit Euler method for time discretization on uniform mesh and a hybrid scheme for spatial discretization on a generalized Shishkin mesh. It is observed that the standard central difference scheme is stable in maximum norm for sufficiently small step size i.e. when $h \leq C \varepsilon$ and the midpoint upwind scheme is of second order convergent outside the layer region. Therefore, we develop the hybrid scheme in such a way that the central difference scheme is applied on the set of indices where it is stable otherwise midpoint upwind scheme is applied. The proposed scheme is parameter uniform with first-order accuracy in time variable and almost second-order accuracy in space variable. \newline
The article is organized as follows: In \textbf{Section 2}, we provide a priori bounds on the derivatives of the solution of the considered problem and also obtain sharper bounds via decomposing the solution into regular and singular components. In \textbf{Section 3}, we construct a generalized piecewise uniform Shishkin mesh and propose a scheme to discretize the considered problem. Also, we discuss the stability of the proposed scheme on the generalized Shishkin mesh. Further, in \textbf{Section 4}, we study and analyse the proposed scheme to prove the $\varepsilon$-uniform convergence of order $O(\Delta t + N^{-2}L^2)$, where $L \sim \log N$. In \textbf{Section 5}, the accuracy of the proposed scheme is demonstrated by conducting and discussing the numerical experiments via tables and graphs. Finally, we end this article with some conclusions given in \textbf{Section 6}. \newline

\textbf{Notations:} Throughout this article, we use C as a generic  positive constant independent of $\varepsilon$ and of the mesh parameters. All the functions defined on the domain Q are measured in supremum norm, denoted by
\begin{align*}
\| f \|_{\overline{Q}} = \sup \limits_{x \in \overline{Q}} |f(x)|.
\end{align*} 
	\section{Analytical Results}
	In this section, the analytical aspects of the solution of the problem (\ref{1.1})-(\ref{1.3}) and its derivatives are studied. The bounds on the solution of the problem (\ref{1.1})-(\ref{1.3}) and its derivatives are derived.\newline
	The operator $L_{\varepsilon}$ satisfies the following minimum principle 
	\begin{lemma}[Minimum principle] \label{lm2.1}
	Let $v$ be a smooth function satisfying $L_{\varepsilon} v(x,t) \leq 0$, $\text{for} ~(x,t) \in Q$ and $v(x,t) \geq 0$, $\text{for} ~(x,t) \in S$, then $v(x,t) \geq 0$, $\text{for} ~ (x,t)\in \overline{Q}$.
	\end{lemma}
	\begin{proof}
	The proof easily follows from \cite{Potter}. 
	\end{proof}
	Using the above minimum principle it can be proved that the solution of the problem (\ref{1.1})-(\ref{1.3}) satisfies the following stability estimate.
	\begin{lemma}\label{lm2.2}
	The exact solution $v(x,t)$ of the problem (\ref{1.1})-(\ref{1.3}) satisfies the following bound
	\begin{align*}
	\| v \|_{\overline{Q}} \leq T \|f\|_{\overline{Q}} + \|v\|_{S}.
    \end{align*}	 
	\end{lemma}
	\begin{proof}
	Defining the following barrier functions
	\begin{align*}
	\Psi^{\pm}(x,t) = t \|f\|_{\overline{Q}} + \|v\|_{S} \pm v(x,t), \quad (x,t) \in Q,
	\end{align*}
	and using the minimum principle we can obtain the required estimate.
	\end{proof}
	Now, we will derive the coarse bounds for the derivatives of the solution $u(x,t)$ of the problem (\ref{1.1})-(\ref{1.3}) on the domains $\overline{Q}_L$ and $\overline{Q}_R$. We will write down the compatibility conditions which will ensure the existence of the unique solution $u(x,t) \in C^{K+\lambda,K/2+\lambda/2}(\overline{Q}_L) \cap C^{K+\lambda,K/2+\lambda/2}(\overline{Q}_R)$. The compatibility conditions for the derivatives $\partial^{k_0}u(x,t)/\partial t^{k_0}$,~$2k_0 \leq K$ on the set of the corner points, e.g., for the case when the initial condition $g(x,0)$ and its derivatives vanish on the set $S^c=(S_r \cup S_l) \cap S_x$ (the set of corner points $(-1,0)$ and $(1,0)$) and $S_0^c=S_0 \cap S_x$ (the point $(0,0)$, is the left corner point of the domain $\overline{Q}_R$ and right corner point of the domain $\overline{Q}_L$) are defined as
	\begin{align} \label{c1}
	\begin{cases}
	\dfrac{\partial^{k_0} g(x,t)}{\partial t^{k_0}}=0, ~\dfrac{\partial^{k} g(x,t)}{\partial x^{k}}=0, ~ 0 \leq k+2k_0 \leq l, \\\\
	\dfrac{\partial^{k+k_0} f(x,t)}{\partial x^{k} \partial t^{k_0}}=0,~ 0 \leq k+2k_0 \leq l-2, ~(x,t) \in S^c \cup S_0^c,
	\end{cases}
	\end{align}
	where $l>0$ is even. By virtue of the conditions (\ref{c1}), we have
	\begin{align} \label{c12}
	\dfrac{\partial^{k_0}}{\partial t^{k_0}}u(x,t)=0,~(x,t) \in S_0^c,~2k_0 \leq l.
	\end{align}
	 Thus, at the point $(0,0)$ the compatibility conditions for the derivatives of $t$ upto order $K/2$ are satisfied, where $K=l$. The compatibility conditions in the corners points of the domains $\overline{Q}_R$ and $\overline{Q}_L$ are fulfilled by the conditions (\ref{c1}) and (\ref{c12}), which are sufficient for $u(x,t) \in C^{K+\lambda,K/2+\lambda/2}(\overline{Q}_L) \cap C^{K+\lambda,K/2+\lambda/2}(\overline{Q}_R)$ (see \cite{MR0241822}).
	\begin{theorem} \label{thm2.1}
	Assume that $a(x,t),~d(x,t),~b(x,t),~f(x,t) \in C^{2+\lambda,1+\lambda/2}(\overline{Q}_R) \cap C^{2+\lambda,1+\lambda/2}(\overline{Q}_L)$, $g(x,t) \in C^{4+\lambda,2 +\lambda /2}(S_x \cap \overline{Q}_R) \cap C^{4+\lambda,2 +\lambda /2}(S_x \cap \overline{Q}_L) \cap C^{4+\lambda,2 +\lambda/2}(\overline{S}_l \cup \overline{S}_r) $ and the compatibility conditions (\ref{c1}) are fulfilled for $K=4$. Then, the solution $u(x,t)$ of the problem (\ref{1.1})-(\ref{1.3}) is such that for all non-negative integers $i,j$, $0\leq i+2j \leq 4$, we have
	\begin{align*}
	\left\| \dfrac{\partial^{i+j}u}{\partial x^i \partial t^j}\right\|_{\overline{Q}} \leq C \varepsilon^{-i}.
	\end{align*}
	\end{theorem}
	\begin{proof}
	We first consider the subdomain $Q_R=(0,1) \times (0,T]$. We handle the layer at $x=1$ by taking the stretched variable $\xi = \dfrac{1-x}{\varepsilon}$. Then, the transformed operator $\widetilde{L}_{\varepsilon}$ defined as
	\begin{align*}
	\widetilde{L}_{\varepsilon} \equiv \left( \dfrac{\partial^2}{\partial \xi^2}  - \widetilde{a} \dfrac{\partial }{\partial \xi} - {\varepsilon} \widetilde{d} \dfrac{\partial }{\partial t} - {\varepsilon} \widetilde{b} \right)\widetilde{u} = \varepsilon \widetilde{f},~\text{on}~\widetilde{Q}_{R},\\
	\widetilde{u}=\widetilde{g},~\text{on}~\widetilde{S}_R,
	\end{align*}
	is not singularly perturbed. Here, we get $\widetilde{Q}_{R} = (0, 1/\varepsilon) \times (0,T]$, $\widetilde{S}_R=\overline{\widetilde{Q}}_R \backslash \widetilde{Q}_R$ and $\widetilde{u}(\xi,t)=u(x,t)$ and similarly the variables $\widetilde{a},~\widetilde{b}$, $\widetilde{d}$ and $\widetilde{f}$ are defined. Now, for each $\zeta \in (0, 1/\varepsilon)$ and each $\delta > 0$, we will define a rectangular neighborhood ${R}_{\zeta,\delta}$ as $((\zeta - \delta, \zeta + \delta) \times (0,T]) \cap \widetilde{Q}_R$ and $\overline{R}_{\zeta,\delta}$ is the closure of ${R}_{\zeta, \delta}$ in the $(\xi,t)$-plane. Using the estimate (10.5) from \cite{MR0241822}, for each point $(\zeta,t) \in \widetilde{Q}_{R}$, we get
	\begin{align*}
	\left \| \dfrac{\partial^{i+j} \widetilde{u}}{\partial \xi^i \partial t^j} \right \|_{\overline{R}_{\zeta,\delta}} \leq C(1+ \left\|\widetilde{u}\right\|_{\overline{R}_{\zeta, 2\delta}}),
	\end{align*}
	where $C$ is independent of the domain $R_{\zeta, \delta}$. Then the above bounds hold for any point $(\zeta,t) \in \widetilde{Q}_R$. Returning to the original variable $x$, we get
    \begin{align*}
    \left\| \dfrac{\partial^{i+j} u}{\partial x^i \partial t^j} \right\|_{\overline{Q}_R} \leq C \varepsilon^{-i}(1+ \|u\|_{\overline{Q}_R}).
    \end{align*}	
    Applying Lemma \ref{lm2.1}, we get the desired estimate on the domain $\overline{Q}_R$. The similar estimate can be obtained on the domain $\overline{Q}_L$, analogously. Hence, we obtain the desired estimate.
	\end{proof}
	
\textbf{Remark:} From the above Theorem \ref{thm2.1} we can easily conclude that $|u_t| \leq C$ and $|u_{tt}| \leq C$. Hereinafter, we will divide the domain $\overline{Q}$ as $\overline{Q}=\overline{Q}_1 \cup \overline{Q}_2 \cup \overline{Q}_3$, where $\overline{Q}_1=[-1,-\delta] \times (0,T]$, $\overline{Q}_2=[-\delta, \delta] \times (0,T]$ and $\overline{Q}_3=[\delta,1] \times (0,T]$, for some $\delta$ lying in the set $(0,1)$.\newline
Before proving the next theorem, we define an operator
\begin{align}
\label{2.2a}
L_j v(x,t) \equiv (\varepsilon v_{xx} + a v_x - d v_t - q_j v)(x,t),
\end{align}
where $q_j(x,t)=b(x,t)-ja_x(x,t)$,~ $1 \leq j \leq 4$. Noticing that $q_j(0,t)=b(0,t)>0$, we can conclude that there exist a positive integer $q_*>0$ such that
\begin{align*}
q_j(x,t) \geq q_*>0, ~ (x,t) \in Q_2.
\end{align*}
The operator $L_j$ satisfies the minimum principle on $Q_2$.
\begin{theorem}\label{thm2.2}
For all non-negative integers $0 \leq i+2j \leq 4$, we have 
\begin{align*}
\left\| \dfrac{\partial^{i+j}u}{\partial x^i \partial t^j} \right\|_{\overline{Q}_2} \leq C.
\end{align*}
\end{theorem}
\begin{proof}
Firstly, we will prove $|u_x|\leq C$ on $[\alpha_1, \beta_1] \subset Q_2$, where $\alpha_1 \in (-\delta,0)$ is such a point that 
\begin{align} \label{2.2}
\left| \dfrac{\partial u(\alpha_1,t)}{\partial x} \right| = \left| \dfrac{u(0,t)-u(-\delta,t)}{\delta} \right| \leq C.
\end{align}
Similarly, $\beta_1 \in (0, \delta)$ is chosen such that 
\begin{align}\label{2.3}
\left| \dfrac{\partial u(\beta_1,t)}{\partial x} \right| \leq C.
\end{align}
Using the bound $|u_t| \leq C$ and Lemma \ref{lm2.2}, we get 
\begin{align}\label{2.4b}
|L_1(u_x)|= |f_x + d_x u_t + b_x u| \leq C.
\end{align}
The inequalities (\ref{2.2}), (\ref{2.3}) and (\ref{2.4b}) together with the fact that $L_1$ satisfies minimum principle on $Q_2$, gives
\begin{align}
\label{2.4c}
 |u_x| \leq C.
\end{align}
Now, we will prove $|u_{xt}| \leq C$. Differentiating eqn. (\ref{1.1}) w.r.t $t$, we get
\begin{align*}
\widetilde{L}w(x,t) \equiv [\varepsilon w_{xx} + a w_x - d w_t -(b+d_t)w](x,t) = (f_t -a_t u_x+b_t u)(x,t),~\text{where}~w=u_t.
\end{align*}
Without loss of generality we can chose $b$ to be so large that $(b+d_t)>0$, as $(b+d_t)>0$ can always be obtained by a preliminary change of variable. As a result, the operator $\widetilde{L}$ satisfies the minimum principle. We will define a new operator $\widetilde{L}_1$ which is same as the operator $L_1$ with $b$ replaced by $b+d_t$ and $f$ replaced by $f_t-a_t u_x+b_t u$. Using the fact that the operator $\widetilde{L}_1$ satisfies minimum principle and following the  arguments used in proving the estimate (\ref{2.4c}), we can show $|w_x|\leq C$ i.e. $|u_{xt}| \leq C$. The estimate for higher derivatives can be obtained on similar lines.
\end{proof}	
For the  error analysis of the proposed numerical scheme we require sharper bounds on the exact solution $u(x,t)$ of the problem (\ref{1.1})-(\ref{1.3}) and its derivatives. Therefore, we decompose the exact solution $u(x,t)$ into regular component $y(x,t)$ and singular component $z(x,t)$ as:
\begin{align*}
u(x,t)=y(x,t)+z(x,t), ~ \text{for} ~ (x,t) \in \overline{Q}.
\end{align*} 
The data of the problem (\ref{1.1})-(\ref{1.3}) is assumed to satisfy,
\begin{align}
\label{c2}
&a(x,t),~b(x,t),~d(x,t),~f(x,t) \in C^{l_0+\lambda}(\overline{Q}_1) \cap C^{l_0+\lambda}(\overline{Q}_3), \nonumber \\
 &g(x,t) \in C^{l_0+\lambda}(S_x \cap \overline{Q}_1) \cap C^{l_0+\lambda}(S_x \cap \overline{Q}_3) \cap C^{l_0+\lambda}(\overline{S}_l \cup \overline{S}_r), ~l_0>0,~ \lambda \in (0,1).
\end{align}
Also, the following compatibility conditions are imposed on the functions $g(x,t)$ and $f(x,t)$,~for $(x,t) \in S^c \cup S_0^c$,
\begin{align}
\label{c3}
\dfrac{\partial^k g(x,t)}{\partial x^k}=0,~\dfrac{\partial^{k_0} g(x,t)}{\partial t^{k_0}}=0,~\dfrac{\partial^{k+k_0} f(x,t)}{\partial x^k \partial t^{k_0}}=0,~ 0 \leq k+k_0 \leq l_0=l_0^*+2.
\end{align}
\begin{theorem}\label{thm2.3}
Assuming sufficient smoothness and compatibility conditions (\ref{c2})-(\ref{c3}) at the corner points, for $K=l_0^*=4$, the smooth component $y(x,t)$ and the singular component $z(x,t)$ satisfies 
\begin{align*}
\left\|\dfrac{\partial^{i+j} y}{\partial x^i \partial t^j}\right\|_{\overline{Q}} \leq C(1+ \varepsilon^{3-i}),
\end{align*}
and 
\begin{align*}
\left\| \dfrac{\partial^{i+j} z}{\partial x^i \partial t^j}\right\|_{\overline{Q}} &\leq C \varepsilon^{-i}( \exp (-\alpha (1+x)/\varepsilon)+\exp (-\alpha (1-x)/\varepsilon)),~\text{for}~0 \leq i+2j \leq 4,
\end{align*}
where $\alpha$ is some constant from $(0,\alpha_0)$.
\end{theorem}	
\begin{proof}
Firstly, we will obtain the bounds on the subdomains $\overline{Q}_1$ and $\overline{Q}_3$.
\newline
The regular component $y(x,t)$ is further expanded in terms of $\varepsilon$ as:
\begin{align}\label{2.4a}
y(x,t)= y_0(x,t)+ \varepsilon y_1(x,t) + \varepsilon^2 y_2(x,t) +r(x,t), ~ (x,t) \in Q_1 \cup Q_3,
\end{align}
where $y_0(x,t)$, $y_1(x,t)$ and $y_2(x,t)$ satisfies the following non-homogeneous hyperbolic equations:
\begin{align}\label{2.4}
\left(a \dfrac{\partial y_0}{\partial x} - d \dfrac{\partial y_0}{\partial t} - b y_0 \right)(x,t)&= f(x,t), ~(x,t) \in Q_1 \cup Q_3,\nonumber \\
y_0(x,t)&= u(x,t), ~  (x,t)\in S_x \cap (\overline{Q}_1 \cup \overline{Q}_3),
\end{align}
\begin{align}\label{2.5}
\left(a \dfrac{\partial y_1}{\partial x} - d \dfrac{\partial y_1}{\partial t} - b y_1 \right)(x,t)&= -\dfrac{\partial^2 y_0(x,t)}{\partial x^2}, ~(x,t) \in Q_1 \cup Q_3, \nonumber \\
y_1(x,t)&= 0, ~ (x,t)\in S_x \cap (\overline{Q}_1 \cup \overline{Q}_3),
\end{align}
\begin{align}\label{2.6}
\left(a \dfrac{\partial y_2}{\partial x} - d \dfrac{\partial y_2}{\partial t} - b y_2 \right)(x,t)&= -\dfrac{\partial^2 y_1(x,t)}{\partial x^2}, ~(x,t) \in Q_1 \cup Q_3, \nonumber \\
y_2(x,t)&= 0, ~ (x,t)\in S_x \cap (\overline{Q}_1 \cup \overline{Q}_3),
\end{align}
and the residue term $r(x,t)$ satisfy
\begin{align}\label{2.7}
L_{\varepsilon} r &= -\varepsilon^3\dfrac{\partial^2 y_2}{\partial x^2}, ~ (x,t) \in Q_1 \cup Q_3, \nonumber \\
r(x,t)&=0,~ (x,t) \in S \cap (\overline{Q}_1 \cup \overline{Q}_3).
\end{align}
The regular component $y(x,t)$ satisfies the following problem
\begin{align}\label{2.8}
L_{\varepsilon}y(x,t)&= f(x,t), ~ (x,t)\in Q_1 \cup Q_3,\nonumber \\
y(x,t)&=u(x,t), ~ (x,t) \in S_x \cap (\overline{Q}_1 \cup \overline{Q}_3),\nonumber \\
y(x,t)&=(y_0+ \varepsilon y_1 + \varepsilon^2 y_2)(x,t), ~  (x,t) \in S_l \cup S_r.
\end{align}
By virtue of the conditions (\ref{c2}) and (\ref{c3}), the data of the problem (\ref{2.8}) are assumed to be sufficiently Smooth  and appropriate compatibility conditions for the data of the problems (\ref{2.4})-(\ref{2.8}) are fulfilled on the sets $S^c$ and $S_0^c$ for the desired smoothness of the components of the expansion (\ref{2.4a}) and to guarantee the inclusion $y(x,t) \in C^{K+\lambda,K/2+\lambda/2}(\overline{Q}_1 \cup \overline{Q}_3),~K=4$. Since $y_0$, $y_1$ and $y_2$ are solutions of first order hyperbolic equations (\ref{2.4})-(\ref{2.6}) whose coefficients are bounded under sufficient compatibility conditions (\ref{c3}), for all non-negative integers $i$, $j$ such that $0 \leq i+2j \leq 4$, we get
\begin{align}\label{2.9}
\left\| \dfrac{\partial^{i+j}y_0}{\partial x^i \partial t^j} \right\|_{\overline{Q}_1 \cup \overline{Q}_3} \leq C, \quad
\left\| \dfrac{\partial^{i+j}y_1}{\partial x^i \partial t^j} \right\|_{\overline{Q}_1 \cup \overline{Q}_3} \leq C \quad \text{and} \quad
\left\| \dfrac{\partial^{i+j}y_2}{\partial x^i \partial t^j} \right\|_{\overline{Q}_1 \cup \overline{Q}_3} \leq C.
\end{align}
As the problem (\ref{2.7}) is similar to the problem (\ref{1.1})-(\ref{1.3}), we can use Theorem \ref{thm2.1} to obtain the following estimate
\begin{align}\label{2.10}
\left\|\dfrac{\partial^{i+j} r}{\partial x^i \partial t^j} \right\|_{\overline{Q}_1 \cup \overline{Q}_3} \leq C \varepsilon^{-i}, ~ \text{for}~ 0 \leq i+2j \leq 4.
\end{align}
 Using the estimates (\ref{2.9})-(\ref{2.10}) and the expansion (\ref{2.4a}), we get
 \begin{align*} 
\left\|\dfrac{\partial^{i+j} y}{\partial x^i \partial t^j}\right\|_{\overline{Q}_1 \cup \overline{Q}_3} \leq C(1+ \varepsilon^{3-i}),~\text{for}~0 \leq i+2j \leq 4.
 \end{align*}
 The singular component satisfies the following homogeneous  initial value problem:
 \begin{align*}
 L_{\varepsilon} z(x,t)&=0,  ~(x,t) \in Q_1 \cup Q_3, \\
 z(x,t)&=0,~ (x,t) \in S_x \cap(\overline{Q}_1 \cup \overline{Q}_3) , \\
 z(x,t)&= u(x,t)-y(x,t), ~ (x,t) \in S_l \cup S_r.
\end{align*}  
Firstly, we will obtain the bounds on the singular component $z(x,t)$ and its derivatives on the domain $\overline{Q}_1$. Let us denote the left singular component by $z_l(x,t)$, where $z_l(x,t)=z(x,t), ~ (x,t) \in \overline{Q}_1$. We construct two barrier functions :
\begin{align*}
\Psi^{\pm}(x,t)= |z_l(-1,t)| \exp \left(\dfrac{-\alpha (x+1)}{\varepsilon}\right) \exp(t) \pm z_l(x,t),
\end{align*}
and it can be easily verified that $\Psi^{\pm}(-1,t) \geq 0$, $\Psi^{\pm}(-\delta,t) \geq 0$, $\Psi^{\pm}(x,0) \geq 0$. Also,
\begin{align*}
L_{\varepsilon} \Psi^{\pm}(x,t) &= |z_l(-1,t)| \exp \left(\dfrac{-\alpha (x+1)}{\varepsilon}\right) \exp(t)\left( \dfrac{\alpha^2}{\varepsilon} - \dfrac{\alpha a}{\varepsilon} -d -b \right)(x,t)\\
&\leq |z_l(-1,t)| \exp \left(\dfrac{-\alpha (x+1)}{\varepsilon}\right) \exp(t)\left( \dfrac{-\alpha}{\varepsilon} (a-\alpha) -d -b \right)(x,t)\\
&\leq |z_l(-1,t)| \exp \left(\dfrac{-\alpha (x+1)}{\varepsilon}\right) \exp(t)\left( \dfrac{-\alpha}{\varepsilon} (\alpha_0 \delta^p-\alpha) -d -b \right)(x,t).
\end{align*}
Since, $a_0(x,t) \geq \alpha_0 >0$, we can always choose $\delta$ in such a way that $\alpha_0 \delta^p \geq \alpha$. Hence, we can conclude that $L_\varepsilon \Psi^{\pm}(x,t) \leq 0$. Using Lemma \ref{lm2.1}, we get $\Psi^{\pm}(x,t) \geq 0$ and hence 
\begin{align}\label{2.12}
|z_l(x,t)| \leq C \exp \left( \dfrac{- \alpha(x+1)}{\varepsilon} \right),~ \text{for} ~ (x,t) \in \overline{Q}_1. 
\end{align}
Using the approach of Theorem \ref{thm2.1}, we can show
    \begin{align*}
    \left\| \dfrac{\partial^{i+j}{z}_l}{\partial x^i \partial t^j} \right\|_{\overline{Q}_1} \leq C \varepsilon^{-i}(1+ \|z_l\|_{\overline{Q}_1}).
    \end{align*}	
    Using estimate (\ref{2.12}), we get
 \begin{align*}
  \left\| \dfrac{\partial^{i+j}{z}_l}{\partial x^i \partial t^j} \right\|_{\overline{Q}_1} \leq C \varepsilon^{-i} \exp (- \alpha(x+1)/{\varepsilon}).
\end{align*}    
Next, the right singular component $z_r(x,t)$, where $z_r(x,t)=z(x,t)$, $(x,t) \in \overline{Q}_3$, and its derivatives on the domain $\overline{Q}_3$ can be handled analogously. Hence, we obtain the following estimate 
\begin{align*}
\left\| \dfrac{\partial^{i+j} z}{\partial x^i \partial t^j}\right\|_{\overline{Q}_1 \cup \overline{Q}_3} &\leq C \varepsilon^{-i}( \exp (-\alpha (1+x)/\varepsilon)+\exp (-\alpha (1-x)/\varepsilon)),~\text{for}~0 \leq i+2j \leq 4.
\end{align*}
The Theorem \ref{thm2.2} guarantees that the solution of the problem (\ref{1.1})-(\ref{1.3}) and its derivatives are smooth in the domain $\overline{Q}_2$. Hence, the desired estimates hold.
\end{proof}
	\section{Discrete Problem}
	In this section, a numerical method is proposed on an appropriate piecewise uniform Shishkin mesh to solve the problem (\ref{1.1})-(\ref{1.3}) numerically. The proposed numerical method consists of implicit Euler method on a uniform mesh to discretize in time variable and a combination of midpoint upwind scheme and central difference scheme on a generalized piecewise uniform Shishkin mesh condensing in the neighborhood of the layer regions to discretize the space variable.
\subsection{Piecewise uniform generalized Shishkin mesh}
We take $M$ and $N$ as the number of intervals in time and space direction, respectively. We construct a rectangular grid defined as $\overline{Q}_{\tau}^{N,M}=\overline{\Omega}_\tau^{N} \times \overline{\Omega}^M$, where $\overline{\Omega}^M$ is a mesh with uniform step-size $\Delta t = T/M$ and $\overline{\Omega}_\tau^{N}$ is a generalized piecewise uniform Shishkin mesh condensing in the neighborhood of the left and right layer regions. We divide the spatial domain $\overline{\Omega}_x=[-1,1]$ into three subdomains such that $\overline{\Omega}_x=\overline{\Omega}_1 \cup \overline{\Omega}_2 \cup \overline{\Omega}_3$ where $\overline{\Omega}_1=[-1,-1 + \tau]$, $ \overline{\Omega}_2= [-1+ \tau, 1-\tau]$ and $\overline{\Omega}_3=[1-\tau,1] $. The subintervals $[-1,-1+\tau]$ and $[1-\tau,1]$ have a uniform mesh with $N/4$ mesh intervals, whereas $[-\tau,\tau]$ have a uniform mesh with $N/2$ mesh intervals. The spatial step sizes $h_i=x_i-x_{i-1}$, for $i=1,2, \ldots,N$ are defined as $h= \dfrac{4 \tau}{N}$ on $ \overline{\Omega}_1 \cup \overline{\Omega}_3$ and $H=\dfrac{4(1-\tau)}{N}$ on $\overline{\Omega}_2 $. The transition parameter $\tau$ is defined as 
\begin{align}
\label{3.1}
\tau = \min \left\lbrace \frac{1}{4}, ~ \tau_0 \varepsilon L \right\rbrace,
\end{align}
where $L$ satisfies $e^{-L} \leq L/N$, $L \leq \ln{N}$ and $ \tau_0 \geq \dfrac{1}{\alpha}$. We can clearly notice that $h=\dfrac{4 \tau_0 \varepsilon L}{N}$ and $\dfrac{2}{N} \leq H \leq \dfrac{4}{N}$. Also, $\overline{\Omega}_\tau^N=\overline{\Omega}_1^N \cup \overline{\Omega}_2^N \cup \overline{\Omega}_3^N $ where $\overline{\Omega}_i^N=\overline{\Omega}_i \cap \overline{\Omega}_\tau^N, ~i=1,2,3$ and $S^{N,M}=S_x^N \cup S_l^N \cup S^N_r$, where $S^N_x=S_x \cap \overline{Q}_\tau^{N,M}$, $S^M_l=S_l \cap \overline{Q}_\tau^{N,M}$, $S^M_r=S_r \cap \overline{Q}_\tau^{N,M}$. For $\tau = 1/4$, uniform mesh is obtained.
\subsection{The finite difference scheme}
Before moving on to the scheme, we define the following finite-difference operators $D_{x}^{+}$, $D_{x}^{-}$, $D_{x}^{0}$, $\delta_{x}^{2}$ and $D_{t}^{-}$ for a given discrete function $v(x_i,t_n)=v_{i}^{n}$ as
\begin{align*}
D_{x}^{+}v_i^n= \frac{v_{i+1}^{n}-v_{i}^{n}}{h_{i+1}}, \quad D_{x}^{-}v_i^n= \frac{v_{i}^{n}-v_{i-1}^{n}}{h_{i}}, \quad D_{x}^{0}v_i^n= \frac{v_{i+1}^{n}-v_{i-1}^{n}}{\widehat{h}_{i}},\\ \delta_{x}^{2}v_i^n= \frac{2(D_{x}^{+}v_{i}^{n}- D_{x}^{-}v_{i}^{n})}{\widehat{h}_{i}} \quad
\text{and} \quad D_{t}^{-}v_i^n= \frac{v_{i}^{n}-v_{i}^{n-1}}{\Delta t},
\end{align*}
where $\widehat{h}_{i}=h_{i}+h_{i+1}$, for $i=0,\ldots,N-1.$\newline
Here, $U_{i\pm1/2}^{n}$ is defined as $U_{i\pm1/2}^{n}= \dfrac{U_{i\pm1}^{n} + U_{i}^{n}}{2}$. For a given function $g(x,t)$ defined on $\overline{Q}_{\tau}^{N,M}$, we define $g_{i\pm1/2}^{n}$ as
\begin{align*}
g_{i\pm1/2}^{n} = \frac{g_{i\pm1}^{n} + g_{i}^{n}}{2}. 
\end{align*}
The proposed scheme is a combination of the central difference scheme 
\begin{align*}
L_{\varepsilon,cen}^{N,M} U_{i}^{n} = \varepsilon \delta_{x}^{2} U_{i}^{n} + a_{i}^{n} D_{x}^{0} U_{i}^{n} - d_i^n D_{t}^{-}U_{i}^{n} - b_{i}^{n} U_{i}^{n},
\end{align*}
and the midpoint upwind scheme
\begin{align*}
L_{\varepsilon,mu,+}^{N,M} U_{i}^{n} = \varepsilon \delta_{x}^{2} U_{i}^{n} + a_{i+1/2}^{n} D_{x}^{+} U_{i}^{n} - d_{i+1/2}^n D_{t}^{-}U_{i+1/2}^{n} - b_{i+1/2}^{n} U_{i+1/2}^{n},\\
L_{\varepsilon,mu,-}^{N,M} U_{i}^{n} = \varepsilon \delta_{x}^{2} U_{i}^{n} + a_{i-1/2}^{n} D_{x}^{-} U_{i}^{n} - d_{i-1/2}^n D_{t}^{-}U_{i-1/2}^{n} - b_{i-1/2}^{n} U_{i-1/2}^{n}.
 \end{align*}
Define a set $I= \{ i \in \{1,\ldots,N-1 \}:  |a_ih_i| < 2\varepsilon\}$, where central difference discretization is stable. We define the  discrete problem corresponding to the problem (\ref{1.1})-(\ref{1.3}) as
\begin{align}
\label{3.2}
L_{\varepsilon}^{N,M}U(x_i,t_n) = \begin{cases}
L_{\varepsilon,cen}^{N,M}U_{i}^{n}&=f_{i}^{n},  ~ \forall~ i \in I,~\text{where}~1\leq i \leq N-1,~n \Delta t \leq T, \\
L_{\varepsilon,mu,+}^{N,M}U_{i}^{n}&=f_{i+1/2}^{n},~ \forall~ i \notin I, ~\text{where}~ 1 \leq i \leq N/2,~n \Delta t \leq T, \\
L_{\varepsilon,mu,-}^{N,M}U_{i}^{n}&=f_{i-1/2}^{n},~ \forall~ i \notin I, ~\text{where} ~ N/2<i \leq N-1,~ n \Delta T \leq T . \\
\end{cases}
\end{align}
Also,
\begin{align}
\label{3.2b}
U(x_i,t_n)&=g(x_{i},t_n), \quad \text{for}~ (x_{i},t_n) \in S^{N,M}.
\end{align}
On rearranging the terms in eqn. (\ref{3.2}), we obtain the following system of equations
\begin{align}
\label{3.3}
\begin{cases}
U_{i}^{0}=g(x_i,t_0), \quad 0 \leq i \leq N,\\
\begin{cases}
r^{-}_{cen,i} U_{i-1}^{n} + r_{cen,i}^{0} U_{i}^{n} + r_{cen,i}^{+} U_{i+1}^{n} &= f^{n}_{cen,i},  ~ \forall~ i \in I,~\text{where}~1\leq i \leq N-1,~n \Delta t \leq T,\\
r^{-}_{mu,i,+} U_{i-1}^{n} + r_{mu,i,+}^{0} U_{i}^{n} + r_{mu,i,+}^{+} U_{i+1}^{n} &= f_{mu,i,+}^{n},~ \forall~ i \notin I, ~\text{where}~ 1 \leq i \leq N/2,~n \Delta t \leq T,\\
r^{-}_{mu,i,-} U_{i-1}^{n} + r_{mu,i,-}^{0} U_{i}^{n} + r_{mu,i,-}^{+} U_{i+1}^{n} &= f_{mu,i,-}^{n},~ \forall~ i \notin I, ~\text{where} ~ N/2<i \leq N-1,~ n \Delta T \leq T,
\end{cases}\\
U_{0}^{n}=g(x_0,t_{n}), \quad U_{N}^{n}=g(x_N,t_{n}),~ n>0.
\end{cases}
\end{align}
Here, for $i\in I$ and $1 \leq i \leq N-1$, the coefficients of the system of equations (\ref{3.3}) are given by 
\begin{align}
\label{3.4}
\begin{cases}
r_{cen,i}^{-}&=\dfrac{2 \varepsilon \Delta t}{\widehat{h}_{i} h_{i}} - \dfrac{a_{i}^{n} \Delta t}{\widehat{h}_{i}},\\
r_{cen,i}^{0}&=\dfrac{- 2 \varepsilon \Delta t}{\widehat{h}_{i}} \left(\dfrac{1}{h_{i}}+ \dfrac{1}{h_{i+1}} \right) - {b_{i}^{n} \Delta t} - d_i^n,\\
r_{cen,i}^{+}&=\dfrac{2 \varepsilon \Delta t}{\widehat{h}_{i} h_{i+1}} + \dfrac{ a_{i}^{n} \Delta t}{\widehat{h}_{i}},\\
f_{cen,i}&= \Delta t f_{i}^{n} - d_i^n U_{i}^{n-1},
\end{cases}
\end{align}
where for $i \notin I $ and $ N/2 + 1\leq i \leq N-1$, the coefficients of the system of equations (\ref{3.3}) are given by
\begin{align}
\label{3.5}
\begin{cases}
r_{mu,i,-}^{+} &= \dfrac{2 \varepsilon \Delta t}{\widehat{h}_{i} h_{i+1}},\\
r_{mu,i,-}^{0} &= \dfrac{-2 \varepsilon \Delta t}{\widehat{h}_{i}} \left( \dfrac{1}{h_{i+1}} + \dfrac{1}{h_{i}} \right) + \dfrac{ a_{i-1/2}^{n} \Delta t}{h_{i}} - \dfrac{d_{i-1/2}^n}{2} - \dfrac{b_{i-1/2}^{n} \Delta t}{2},\\
r_{mu,i,-}^{-} &= \dfrac{2 \varepsilon \Delta t}{\widehat{h}_{i} h_i} - \dfrac{  a_{i-1/2}^{n} \Delta t}{h_{i}} - \dfrac{d_{i-1/2}^n}{2} - \dfrac{b_{i-1/2}^{n} \Delta t}{2},\\
f_{mu,i,-}&= \Delta t f_{i-1/2}^{n} - d_{i-1/2}^n U_{i-1/2}^{n-1},
\end{cases}
\end{align}
and for $i \notin I$ and $1 \leq i \leq N/2$, the coefficients of the system of equations (\ref{3.3}) are given by 
\begin{align}
\label{3.6}
\begin{cases}
r_{mu,i,+}^{+}&= \dfrac{2 \varepsilon \Delta t}{\widehat{h}_{i} h_{i+1}} + \dfrac{ a_{i+1/2}^{n} \Delta t}{h_{i+1}} - \dfrac{d_{i+1/2}^n}{2} - \dfrac{\Delta t b_{i+1/2}^{n}}{2},\\
r_{mu,i,+}^{0}&= \dfrac{-2 \varepsilon \Delta t}{\widehat{h}_{i}}\left( \dfrac{1}{h_i} + \dfrac{1}{h_{i+1}} \right)  - \dfrac{ a_{i+1/2}^{n} \Delta t}{h_{i+1}} - \dfrac{d_{i+1/2}^n}{2} - \dfrac{\Delta t b_{i+1/2}^{n}}{2},\\
r_{mu,i,+}^{-}&= \dfrac{2 \varepsilon \Delta t}{\widehat{h}_{i} h_{i}},\\
f_{mu,i,+}&= \Delta t f_{i+1/2}^{n} - d_{i+1/2}^n U_{i+1/2}^{n-1}.
\end{cases}
\end{align}
In the next lemma, we will prove that the tridiagonal matrix corresponding to the system (\ref{3.3}) is an M-matrix. For this lemma we define a positive integer, $\kappa \geq \alpha/2 > 0$. We know that $ a(x_{N/2+1},t_n)<0$ and $  a(x_{N/2-1},t_n)>0$. Also, $a_0(x,t) \geq \alpha_0 > 0$, so there exists a constant $\kappa > 0$ such that $ |a(x_i,t_n)| \geq \kappa > 0$, for $i > N/2$ and $i < N/2$.
\begin{lemma}
	\label{lm3.1}
	Assume that there exists some $N_0>0$ satisfying 
	\begin{align}
	\label{3.7}
	\begin{cases}
     N_0 \kappa \geq 2 \left(  \dfrac{\| d \|_{\overline{Q}}}{\Delta t} + \| b \|_{\overline{Q}} \right) \quad and\\
	2 \tau_0 \| a \|_{\overline{Q}} < \dfrac{N_0}{\ln N_0},
	\end{cases}
	\end{align}
	such that for all $N \geq N_0$, the tridiagonal matrix corresponding to difference operator (\ref{3.3}) is an M-matrix.
\end{lemma}
\begin{proof}
	Firstly, we will consider the case of $a(x,t) \geq 0$ i.e $(x,t) \in (-1,0] \times (0,T]$.\newline
	For $i \in I$, we have $ |a_i h_i| < 2 \varepsilon$, which gives 
	\begin{align*}
	r_{cen,i}^{-} = \dfrac{2 \varepsilon \Delta t}{ \widehat{h}_i h_i} - \dfrac{\Delta t a_i}{\widehat{h}_i}= \dfrac{\Delta t}{\widehat{h}_i} \left( \dfrac{2 \varepsilon}{h_i} -  a_i\right) > 0,
	\end{align*}
	and
	\begin{align*}
	r_{cen,i}^{+}= \dfrac{2 \varepsilon \Delta t}{\widehat{h} h_{i+1}} + \dfrac{ a_i \Delta t}{\widehat{h}_i} > 0.
	\end{align*}
	Using (\ref{3.4}), we get
	\begin{align*}
	|r_{cen,i}^{+}| + |r_{cen,i}^{-}| < |r_{cen,i}^{0}|.
	\end{align*}
	For all $N \geq N_0$, where $N_0$ satisfies $2 \tau_0 \| a \|_{\overline{Q}} < \dfrac{N_0}{\ln N_0}$, we have $| a_i h_i | < 2 \varepsilon$, for all $ x_i \in \overline{\Omega}_1^N$, which implies the set $\{ 1, \ldots ,N/4 \}  \subseteq I$. From this we can conclude that $L_{\varepsilon,mu,+}^{N,M}$ is applied for $N/4< i \leq N/2$, where $i \notin I$. \newline
	For $i \notin I$, we have 
	\begin{align*}
	r_{mu,i,+}^{+} &= \dfrac{2 \varepsilon \Delta t}{ \widehat{h_{i}} h_{i+1}} + \dfrac{ a_{i+1/2}^{n} \Delta t}{h_{i+1}} - \dfrac{d_{i+1/2}^n}{2} - \dfrac{\Delta t b_{i+1/2}^{n}}{2},\\
	&> \Delta t \left( \dfrac{ a_{i+1/2}^{n}}{h_{i+1}} - \dfrac{d_{i+1/2}^n}{2 \Delta t} - \dfrac{b_{i+1/2}}{2} \right).
	\end{align*}
	 Using $\dfrac{1}{H} = \dfrac{N}{4(1- \tau)} \geq \dfrac{N_0}{4}$ and assumption (\ref{3.7}), we get $r_{mu,i,+}^{+} > 0$. Clearly, $r_{mu,i,+}^{-} > 0$. Also $|r_{mu,i,+}^{+}| + |r_{mu,i,+}^{-}| < |r_{mu,i,+}^{0}|$.\newline
	 Similarly, the case of $a(x,t)<0$ can be proved analogously. This completes the proof.
 \end{proof}

\begin{lemma}[Discrete minimum principle]
	\label{lm3.2}
	Let $W^{N}$ be any mesh function defined on $\overline{Q}^{N,M}_{\tau}$. If $W^{N}(x_i,t_n) \geq 0$, $\text{for} ~ (x_i,t_n) \in S^{N,M}$ and $L_{\varepsilon}^{N,M} W^{N} (x_i,t_n) \leq 0$, $\text{for} ~ (x_i,t_n) \in Q_{\tau}^{N,M}$, then $W^{N}(x_i,t_n) \geq 0$, $\text{for} ~ (x_i,t_n) \in \overline{Q}_{\tau}^{N,M}. $
\end{lemma}

\begin{lemma}
\label{lm3.3}
Let $ W^{N} $ be any mesh function defined on $Q_{\tau}^{N,M}$. If $W^{N}(x_i,t_n) \geq 0$, $\text{for} ~ (x_i,t_n) \in S^{N,M}$ then
\begin{align*}
|W^{N}(x_i,t_n)| \leq \max \limits_{S^{N,M}} | W^{N} | + \dfrac{T}{\beta} \max \limits_{S^{N,M}} |L_{\varepsilon}^{N,M} W^{N}|, \quad for ~ (x_i,t_n) \in \overline{Q}^{N,M}.
\end{align*} 
\end{lemma}
\begin{proof}
	Using the barrier function
	 \begin{align*}
	 \phi^{\pm} = \max |W^{N}| + \dfrac{t_n}{\beta} \max |L_{\varepsilon}^{N,M}W| \pm W^{N},
	\end{align*}
	and the discrete minimum principle we obtain the desired estimate.
\end{proof}
\section{Convergence Analysis}
In this section, we separately prove the error bounds for the regular and the singular components. The $\varepsilon$-uniform error estimate is obtained by combining the error bounds for the regular and singular components.

\begin{lemma}
	\label{lm4.1}
	Let $v(x,t)$ be a smooth function defined on the domain $\overline{Q}$ and $v_i^n=v(x_i,t_n)$ be the corresponding discrete function on $\overline{Q}_{\tau}^{N,M}$. Then the local truncation error at the mesh points $(x_i,t_n)$ corresponding to the discrete scheme $L_{\varepsilon}^{N,M}$ satisfies
	\begin{align*}
	\tau_i^{n}&= \begin{cases}
	|L_{\varepsilon,cen}^{N,M} v_i^n - L_{\varepsilon} v(x_i,t_n)|, \quad for ~ i \in I, \\
	|L_{\varepsilon,mu,+}^{N,M} v_i^n - L_{\varepsilon} v(x_{i+1/2},t_n)|, \quad for ~ i \notin I,~ i \leq N/2, \\
	|L_{\varepsilon,mu,-}^{N,M} v_i^n - L_{\varepsilon} v(x_{i-1/2},t_n)|, \quad for ~ i \notin I,~ i> N/2, 
	\end{cases}\\
	& \leq \begin{cases}
   C \left[\Delta t + h_i  \varepsilon \int\limits_{x_{i-1}}^{x_{i+1}} |v_{xxxx}| dx+  h_i  \int\limits_{x_{i-1}}^{x_{i+1}} |v_{xxx}|dx \right], \quad for ~ i\in I, \\\\
   C \left[\Delta t + \varepsilon \int\limits_{x_{i-1}}^{x_{i+1}} |v_{xxx}|dx + h_{i+1}\int\limits_{x_{i}}^{x_{i+1}}\left(|v_{xxx}|+|v_{xx}|+|v_x|\right)dx \right], \quad for ~ i\not\in I,~ i \leq N/2,\\\\
   C \left[\Delta t + \varepsilon \int\limits_{x_{i-1}}^{x_{i+1}} |v_{xxx}|dx+ h_{i}\int\limits_{x_{i-1}}^{x_{i}}\left(|v_{xxx}|+|v_{xx}|+|v_x|\right)dx \right], \quad for ~ i\not\in I,~ i > N/2.
	\end{cases}
	\end{align*}
\end{lemma}
\begin{proof}
	The proof of the above lemma follows easily from the Taylor's series expansion.
\end{proof}
The next two lemmas provide certain barrier functions which will be required to prove the $\varepsilon$-uniform error estimates on the  singular component. We construct the following barrier functions:
\begin{align}
\label{4.1b}
\bar{B}_i = \begin{cases}
\left(1+\dfrac{\alpha h}{{\varepsilon}}\right)^{-i}, \quad for ~ i=0,\ldots,N/4,\\\\
\left(1+\dfrac{\alpha h}{{\varepsilon}}\right)^{-N/4}, \quad for ~ i=N/4,\ldots,N/2,
\end{cases}
\end{align}
\begin{align}
\label{4.1}
\hat{B}_i = \begin{cases}
\left(1+\dfrac{\alpha h}{{\varepsilon}}\right)^{-N/4}, \quad for ~ i=N/2,\ldots,3N/4,\\\\
\left(1+\dfrac{\alpha h}{{\varepsilon}}\right)^{-(N-i)}, \quad for ~ i=3N/4,\ldots,N.
\end{cases}
\end{align}

\begin{lemma}
 	\label{lm4.2}
 	There exist discrete functions $\bar{B}_i$ and $\hat{B}_i$ such that for $1 \leq i \leq N-1$, we have
 	\begin{align*}
 	L_{\varepsilon}^{N,M} \bar{B}_i \leq \begin{cases}
 	\dfrac{-C }{\varepsilon}\bar{B}_i , \quad for ~ i=1,\ldots,N/4,\\\\
 	{-C \bar{B}_i}, \quad for ~ i=N/4+1,\ldots,N/2, 
 	\end{cases}
 	\end{align*}
 	\begin{align*}
 	L_{\varepsilon}^{N,M} \hat{B}_i \leq \begin{cases}
 	 {-C \hat{B}_i}, \quad for ~ i=N/2+1,\ldots,3N/4,\\\\
 	\dfrac{-C }{\varepsilon}\hat{B}_i , \quad for ~ i=3N/4+1,\ldots,N-1. 
 	\end{cases}
 	\end{align*}
\end{lemma}

\begin{proof}
	We first deduce the result for $L_\varepsilon^{N,M} \hat{B}_i$. Consider the case for $i \geq 3N/4+1$. Since, the set of points $\{3N/4+1, \ldots ,N-1\} \subset I$, the operator $L_{\varepsilon,cen}^{N,M}$ is applied on the barrier function $\hat{B}_{i}$, to get
	\begin{align}
	\label{4.2}
	L_{\varepsilon,cen}^{N,M} \hat{B}_i =  r_i^{-} \hat{B}_{i-1}+ r_i^{0} \hat{B}_{i}  + r_i^{+} \hat{B}_{i+1} ,
	\end{align}
	where
	\begin{align*}
	r_{i}^{-} = \dfrac{2 \varepsilon}{h_i \widehat{h}_i} - \dfrac{a_i}{\widehat{h}_i}, \quad r_{i}^{+} = \dfrac{2 \varepsilon}{h_{i+1} \widehat{h}_i} + \dfrac{ a_i}{\widehat{h}_i},\\
	r_{i}^{0}= -\dfrac{2 \varepsilon}{h_i \widehat{h}_i}- \dfrac{2 \varepsilon}{h_{i+1} \widehat{h}_i} - b_i^n. 
	\end{align*}
	On simplifying the eqn. (\ref{4.2}), we obtain
	\begin{align*}
	L_{\varepsilon,cen}^{N,M} \hat{B}_{i}  &\leq - \hat{B}_{i} \left[ \dfrac{\alpha}{2{h}} \left\lbrace \dfrac{h}{\varepsilon+ \alpha h}(-a_i^n-2 \alpha)- \dfrac{a_i^n h}{\varepsilon} \right\rbrace + b_i^n  \right].
	\end{align*}
	Since, $a_i^n<0$ for $i > N/2$, we get
	\begin{align*}
	L_{\varepsilon}^{N,M} \hat{B}_i \leq - \dfrac{C}{\varepsilon} \hat{B}_i.
	\end{align*}
	For $N/2+1 \leq i \leq 3N/4$, whether $i \in I$ or $i \notin I$ we can easily get
	\begin{align*}
	L_{\varepsilon}^{N,M}\hat{B}_i = - b(x_i,t_n) \hat{B}_i \leq -C \hat{B}_i.
	\end{align*}
	The result for $L_{\varepsilon}^{N,M} \bar{B}_i$ can be established  analogously.
\end{proof}
\begin{lemma}\label{lm4.3}
The discrete barrier functions $\bar{B}_i$ and $\hat{B}_i$ satisfy the following estimates
\begin{enumerate}
\item[(i)]
 $\exp(-\alpha(1+x_i)/\varepsilon) \leq \bar{B}_i$ and  $\exp(-\alpha(1-x_i)/\varepsilon) \leq \hat{B}_i$,
\item[(ii)]
$
 \bar{B}_i \leq \begin{cases}
1, ~ \text{for}~ 1 \leq i \leq N/4,\\
C N^{-\alpha \tau_0} L^{\alpha \tau_0}, ~ \text{for}~ N/4 < i \leq N/2,
\end{cases}
$  \newline and \newline
$
 \hat{B}_i \leq \begin{cases}
C N^{-\alpha \tau_0} L^{\alpha \tau_0}, ~ \text{for}~ N/2 < i \leq 3N/4,\\
1, ~ \text{for}~ 3N/4 < i \leq N-1.
\end{cases}
$
\end{enumerate}
\end{lemma}
\begin{proof}
To prove the estimates (i) and (ii) one mainly uses the inequalities $e^{-t} \leq (1+t)^{-1}$ and $(1+t)^{-1} \leq e^{t+t^2}$, respectively, where $t \geq 0$. For more details one can refer to [\cite{becher},~Appendix].
\end{proof}
We will split the discrete problem (\ref{3.2})-(\ref{3.2b}) into left and right discrete problems centred around $x_{N/2}=0$. The problems are defined as 
\begin{align}
\label{4.3a}
L_\varepsilon^{N,M}U_L(x_i,t_n)&=f_i^n, ~ (x_i,t_n) \in Q_{L,\tau}^{N,M}=Q_{\tau}^{N,M} \cap Q_L , \nonumber \\
U_L(-1,t_n)&=g(-1,t_n), ~ for ~ n \Delta t \leq T, \nonumber\\
U_L(0,t_n)&=u(0,t_n), ~ for ~ n \Delta t \leq T,\nonumber\\
U_L(x_i,0)&=g(x_i,0),~for~ i \leq N/2,
\end{align}
and
\begin{align}
\label{4.3b}
L_\varepsilon^{N,M}U_R(x_i,t_n)&=f_i^n, ~ (x_i,t_n) \in Q_{R,\tau}^{N,M}=Q_{\tau}^{N,M} \cap Q_R , \nonumber \\
U_R(1,t_n)&=g(1,t_n), ~ for ~ n \Delta t \leq T, \nonumber\\
U_R(0,t_n)&=u(0,t_n), ~ for ~ n \Delta t \leq T,\nonumber\\
U_R(x_i,0)&=g(x_i,0),~for~ i \geq N/2,
\end{align} 
where $ U(x_i,t_n) = \begin{cases} U_L(x_i,t_n), ~ i \leq N/2,~ n \Delta t \leq T,\\
U_R(x_i,t_n), ~ i \geq N/2,~ n \Delta t \leq T \end{cases}$.
To obtain $\varepsilon$-uniform error estimates, we decompose each $U_{L \backslash R}$ into a regular components $Y_{L \backslash R}$ and a singular components $Z_{L \backslash R}$ as
\begin{align}
\label{4.3}
U_{L \backslash R}(x_i,t_n) = (Y_{L \backslash R} + Z_{L \backslash R})(x_i,t_n), \quad (x_i,t_n) \in \overline{Q}_{\tau}^{N,M},
\end{align}
where the discrete left and right regular components $Y_{L}(x_i,t_n)$ and $Y_{R}(x_i,t_n)$, respectively, are the solutions of the problems
\begin{align}
\label{4.4}
L_{\varepsilon}^{N,M} Y_L(x_i,t_n) &= f(x_i,t_n), \quad \forall~ (x_i,t_n) \in Q_{L,\tau}^{N,M},\nonumber \\
Y_L(x_i,t_n)&=y(x_i,t_n), \quad \forall~(x_i,t_n) \in S_l^{M} ,\nonumber \\
Y_L(x_i,t_n)&=u(x_i,t_n), \quad \forall~(x_i,t_n) \in (S_x^{N} \cap Q_{L,\tau}^{N,M}) \cup \{(x_{N/2},t_n), ~ n \Delta t \leq T\},
\end{align}
and
\begin{align}
\label{4.4a}
L_{\varepsilon}^{N,M} Y_R(x_i,t_n) &= f(x_i,t_n), \quad \forall~ (x_i,t_n) \in Q_{R,\tau}^{N,M},\nonumber \\
Y_R(x_i,t_n)&=y(x_i,t_n), \quad \forall~(x_i,t_n) \in S_r^{M},\nonumber \\
Y_R(x_i,t_n)&=u(x_i,t_n), \quad \forall~(x_i,t_n) \in (S_x^{N} \cap Q_{R,\tau}^{N,M}) \cup \{(x_{N/2},t_n), ~ n \Delta t \leq T\} .
\end{align}
The discrete left and right singular components $Z_{L}$ and $Z_{R}$, respectively, are the solutions of the problems 
\begin{align} \label{4.5}
L_{\varepsilon}^{N,M} Z_L(x_i,t_n) &= 0, \quad \forall~ (x_i,t_n) \in Q_{L,\tau}^{N,M},\nonumber \\
Z_L(x_i,t_n)&=z_l(x_i,t_n), \quad \forall~(x_i,t_n) \in S_l^{M} ,\nonumber \\
Z_L(x_i,t_n)&=0, \quad \forall~(x_i,t_n) \in (S_x^{N} \cap Q_{L,\tau}^{N,M}) \cup \{(x_{N/2},t_n), ~ n \Delta t \leq T\},
\end{align}
and
\begin{align}
\label{4.6}
L_{\varepsilon}^{N,M} Z_R(x_i,t_n) &= 0, \quad \forall~ (x_i,t_n) \in Q_{R,\tau}^{N,M},\nonumber \\
Z_R(x_i,t_n)&=z_r(x_i,t_n), \quad \forall~(x_i,t_n) \in S_r^{M} ,\nonumber \\
Z_R(x_i,t_n)&=0, \quad \forall~(x_i,t_n) \in (S_x^{N} \cap Q_{R,\tau}^{N,M}) \cup \{(x_{N/2},t_n), ~ n \Delta t \leq T\} .
\end{align}
\begin{lemma}[Error estimate for the regular component]
	\label{lm4.4}
	Under the assumptions (\ref{3.7}), the following error estimate is satisfied by the smooth component at each mesh points $(x_i,t_n)  \in  \overline{Q}_{\tau}^{N,M}$,
	\begin{align*}
	|(Y-y)(x_i,t_n)| \leq C(\Delta t + N^{-2}), \quad \mbox{for} ~ 0 \leq i \leq N, ~ n \Delta t \leq T.
	\end{align*}
\end{lemma}
\begin{proof}
	Firstly, the result is deduced for the right regular component $Y_R$. We will consider two cases depending upon the relationship between $\varepsilon$ and $N$.
	\begin{enumerate}
	 \item[\textbf{Case I:}] When $\varepsilon > \|a\|_{\infty}/N$. In this case we have $|a_i h_i| < 2 \varepsilon$, for all $i \in \{N/2+1,\ldots,N-1\} $ which implies the set $\{N/2+1, \ldots, N-1\} \subseteq I$. We get
 \begin{align*}
 |L_{\varepsilon}^{N,M}(Y_R-y)(x_i,t_n)| &= |L_{\varepsilon,cen}^{N,M}(Y_R-y)(x_i,t_n)|,\quad \forall  ~N/2+1\leq i \leq N-1,\thinspace n\Delta t \leq T  \\
& \leq C[\Delta t+h_i(h_{i+1}+h_i)(\varepsilon|y_{xxxx}|+|y_{xxx}|)].
\end{align*} 
Using the fact that $ h_{i+1}+h_i\leq 2N^{-1} $, for all $i$ and the bounds on the derivatives of $ y $ given in Theorem \ref{thm2.3}, we get
  \begin{align*}
 |L_{\varepsilon}^{N,M}(Y_R-y)(x_i,t_n)|\leq C(\Delta t+ N^{-2}),\quad N/2+1 \leq i \leq N-1,\thinspace n\Delta t \leq T.
 \end{align*}  
 Applying Lemma \ref{lm3.3}, we can obtain 
 \begin{align}
 \label{t1}
 |(Y_R-y)(x_i,t_n)|\leq C(\Delta t+ N^{-2}),\quad N/2 \leq i \leq N,\thinspace n\Delta t \leq T.
 \end{align}

	\item[\textbf{Case II:}]
	When $\varepsilon \leq \|a\|_{\infty}/N$. In this case the truncation error associated with the smooth component for $i > N/2$ is given by
	\begin{align*}
	|L_{\varepsilon}^{N,M}(Y_R-y)(x_i,t_n)| \leq \begin{cases}
	C \Delta t \left[ h_i(h_i + h_{i+1})(\varepsilon |y_{xxxx}| + |y_{xxx}|)\right], \quad \text{for} ~ i \in I,\\
	C \Delta t \left[ \varepsilon(h_{i+1} + h_i)|y_{xxx}| + h_{i}^2( |y_{xxx}| + |y_{xx}| + |y_{x}|) \right], \quad \mbox{for} ~ i \notin I.
	\end{cases}
	\end{align*}
	Applying the bounds given in Theorem~\ref{thm2.3} and using the fact that $ h_{i+1}+h_i\leq 2N^{-1} $ and $\varepsilon \leq C N^{-1}$, we obtain
	\begin{align*}
	|L_{\varepsilon}^{N,M}(Y_R-y)(x_i,t_n)| \leq C (\Delta t + N^{-2}), \quad N/2+1 \leq i \leq N-1.
	\end{align*}
	Now, using the Lemma~\ref{lm3.3}, we get 
	\begin{align}
	\label{4.7}
	|(Y_R-y)(x_i,t_n)| \leq C (\Delta t + N^{-2}), \quad for ~ N/2\leq i \leq N ,~n\Delta t \leq T.
	\end{align}
	Similarly, for $i \leq N/2$, we can show  
	\begin{align}
	\label{4.8}
	|(Y_L-y)(x_i,t_n)| \leq C (\Delta t + N^{-2}), \quad for ~ 0 \leq i\leq N/2 ,~n\Delta t \leq T.
	\end{align}
	Combining eqns. (\ref{4.7})-(\ref{4.8}), we get
	\begin{align*}
	|(Y-y)(x_i,t_n)| \leq C (\Delta t + N^{-2}), \quad for ~ (x_i,t_n) \in \overline{Q}_\tau^{N,M}.
	\end{align*}
	\end{enumerate}
\end{proof}
\begin{lemma}[Error estimates for the singular component]
	\label{lm4.6}
	Under the assumption (\ref{3.7}), the following error estimate is satisfied by the singular component $Z(x_i,t_n)$ at each mesh points $(x_i,t_n) \in \overline{Q}_{\tau}^{N,M}$
	\begin{align*}
	|(Z-z)(x_i,t_n)| \leq C(\Delta t + N^{-2}L^{2}), \quad for ~ 0 \leq i \leq N,~ n \Delta t \leq T.
	\end{align*}
\end{lemma}
\begin{proof}
	Firstly, consider the right boundary layer component $Z_R$. To start with, we first compute the error in the outer region $ [0,1-\tau] \times (0,T] $ and then, we analyse the error in the inner region $(1-\tau,1] \times (0,T]$. In the outer region, both $Z_R$ and $z_r$ are small irrespective of the fact whether $i \in I$ or $i \notin I $. So, we use the following barrier functions
	\begin{align}
	\label{4.13}
	\Phi_{R,i}^{n,\pm} = C \hat{B}_i \pm Z_{R}(x_i,t_n), \quad (x_i,t_n) \in \overline{Q}^{N,M}_{R,\tau},
	\end{align}
	where $C=|z_r(x_N,t_n)|$. Using Lemma \ref{lm4.2}, we get
	\begin{align*}
	L_{\varepsilon}^{N,M} \Phi^{n,\pm}_{R,i} = C L_{\varepsilon}^{N,M} \hat{B}_i \pm L_{\varepsilon}^{N,M} Z_R(x_i,t_n) \leq 0.	
	\end{align*}
	Also, $\Phi_{R,N/2}^{n,\pm} \geq 0 $ and $\Phi_{R,i}^{0,\pm} \geq 0 $, for $N/2 \leq i \leq N, ~ n \Delta t \leq T$, and $\Phi_{R,{N}}^{n,\pm}=C \pm Z_R(x_N,t_n) \geq 0$, $n \Delta t \leq T$.Using discrete minimum principle, we get
	\begin{align}
	\label{4.14}
	|Z_{R}(x_i,t_n)| \leq C \hat{B}_i, \quad for ~ N/2 \leq i \leq N,~ n \Delta t \leq T.
	\end{align}
	Now, for $N/2 \leq i \leq 3N/4$, inequality (\ref{4.14}) and Lemma \ref{lm4.3} results into
	\begin{align}
	\label{4.15}
	|Z_{R}(x_i,t_n)| \leq C L^{\alpha \tau_0}N^{-\alpha \tau_0} .
	\end{align}
	From Theorem \ref{thm2.3} and Lemma \ref{lm4.3}, we obtain 
	\begin{align}
	\label{4.16}
	|z_r(x_i,t_n)| &\leq \exp \left( \dfrac{-\alpha(1-x_i)}{{\varepsilon}} \right) \leq C \hat{B}_i \leq C L^{\alpha \tau_0}N^{-\alpha \tau_0}.
	\end{align}
	Using (\ref{4.15}) and (\ref{4.16}), we can obtain the following bound in the outer region $[0,1-\tau] \times (0,T]$
	\begin{align}
	\label{4.16a}
	|Z_R(x_i,t_n) - z_r(x_i,t_n)| &\leq |Z_R(x_i,t_n)| + |z(x_i,t_n)|\nonumber \\
	& \leq C N^{-2} L^2,\quad \forall~  N/2 \leq i \leq 3N/4, ~ n\Delta t \leq T,
	\end{align}  
	with the choice of $\tau_0=2/\alpha$. \newline
Similarly, using the same arguments as in the previous case we can show that the left singular component $Z_L$ in the outer region $[-1+\tau,0] \times (0,T]$, satisfies the error bound   
\begin{align}
\label{4.16b}
|Z_{L}(x_i,t_n)-z_l(x_i,t_n)| \leq C N^{-2} L^2,\quad \forall~  N/4 \leq i \leq N/2, ~ n\Delta t \leq T.
\end{align}
Next, we consider the inner region $(1-\tau,1] \times (0,T]$. Since for all $N \geq N_0$ satisfying condition (\ref{3.7}), we have $ |a_i h_i| < 2 \varepsilon$, for all $i=3N/4+1,\ldots,N$, therefore $\{3N/4+1,\ldots,N\} \subset I$. Hence, $L_{\varepsilon,cen}^{N,M}$ is applied in the layer region $(1-\tau,1] \times (0,T]$.\newline
 For $3N/4+1\leq i \leq N-1$, we have 
\begin{align}
\label{4.20}
|L_{\varepsilon,cen}^{N,M}(Z_{R}-z_r)(x_i,t_n)| &\leq C \left[ \Delta t + h_i \int_{x_{i-1}}^{x_{i+1}} \varepsilon |(z_r)_{xxxx}| + |(z_r)_{xxx}| dx \right] \nonumber \\ 
& \leq C \left[ \Delta t + \dfrac{h_i}{\varepsilon^3} \int_{x_{i-1}}^{x_{i+1}} \exp \left( \dfrac{-\alpha(1-x)}{\varepsilon} \right) dx \right]\nonumber \\
&= C \left[ \Delta t + \dfrac{h_i}{\varepsilon^2 \alpha} \left\{ \exp \left( \dfrac{-\alpha (1-x_{i+1})}{\varepsilon} \right) - \exp \left( \dfrac{-\alpha (1-x_{i-1})}{\varepsilon} \right)  \right\} \right] \nonumber \\
&= C \left[ \Delta t + \dfrac{h_i}{\varepsilon^2 \alpha} \exp \left( \dfrac{-\alpha (1-x_{i})}{\varepsilon} \right) \left\{ \exp \left( \dfrac{\alpha h}{\varepsilon} \right) - \exp \left( \dfrac{-\alpha h}{\varepsilon} \right)  \right\} \right]\nonumber \\
&= C \left[ \Delta t + \dfrac{h_i}{\varepsilon^2 \alpha} \exp \left( \dfrac{-\alpha (1-x_{i})}{\varepsilon} \right) \sinh \left( \dfrac{\alpha h}{\varepsilon} \right) \right].
\end{align}
Under the assumption $2 \tau_0 \| a \|_{\overline{Q}} < \dfrac{N_0}{\ln N_0}$, we have $\alpha h/ {\varepsilon}< 2$. Also, we know that $\sinh {\xi} \leq C \xi$, for $0 \leq \xi \leq 2$, so we have $\sinh \left(\dfrac{\alpha h}{\varepsilon}\right) \leq C {\dfrac{\alpha h}{\varepsilon}}$. \newline
From inequality (\ref{4.20}), we get
\begin{align*}
|L_{\varepsilon,cen}^{N,M}(Z_{R}-z_r)(x_i,t_n)| &\leq C \left[ \Delta t + \dfrac{h^2}{\varepsilon^{3}} \exp \left( \dfrac{- \alpha(1- x_i)}{ \varepsilon } \right) \right] \nonumber \\
&=C \left[ \Delta t + \dfrac{N^{-2}L^2}{\varepsilon} \exp \left( \dfrac{- \alpha(1- x_i)}{ \varepsilon } \right) \right].
\end{align*}
Using Lemma \ref{lm4.3}, we get
\begin{align}
\label{4.21}
|L_{\varepsilon,cen}^{N,M}(Z_{R}-z_r)(x_i,t_n)| 
&\leq C \left[ \Delta t + \dfrac{N^{-2}L^2}{\varepsilon} \hat{B}_i \right].
\end{align}
From inequality (\ref{4.16a}), we get
\begin{align*}
|(Z_R-z_r)(x_{3N/4},t_n)| \leq C N^{-2}L^{2}.
\end{align*}
Also, $|(Z_R-z_r)(x_N,t_n)|=|(Z_R-z_r)(x_i,t_0)|=0$, for $3N/4 \leq i \leq N,~n \Delta t \leq T$. Using the estimate given in eqn. (\ref{4.21}), we construct the following barrier functions
\begin{align*}
\psi^{\pm}(x_i,t_n)= C(N^{-2}L^{2} \hat{B}_i + (\Delta t + N^{-2}L^2) t_n ) \pm (Z_R-z_r)(x_i,t_n).
\end{align*}
It can be observed that $\psi^{\pm}(x_i,t_n) \geq 0$ at the points  $(x_{3N/4},t_n)$, $(x_{N},t_n)$ and $(x_i,t_0)$ in $S^{N,M}$. Using Lemma \ref{lm4.2} and eqn. (\ref{4.21}), we get
\begin{align*}
L_{\varepsilon,cen}^{N,M} \psi^{\pm}(x_i,t_n) = &C( N^{-2} L^{2} L_{\varepsilon,cen}^{N,M} \hat{B}_{i}(x_it_n) - d(x_i,t_n)(\Delta t + N^{-2}L^{2}) ) \\
&\pm L_{\varepsilon,cen}^{N,M} (Z_R-z_r)(x_i,t_n) \leq 0, ~ for~ 3N/4+1 \leq i \leq N-1,~ n \Delta t \leq T, 
\end{align*} 
On using the discrete minimum principle, we get the following estimate 
\begin{align*}
|(Z_R-z_r)(x_i,t_n)| &\leq C(N^{-2}L^{2} \hat{B}_i + \Delta t),~for~3N/4 \leq i \leq N,~n \Delta t \leq T \\
& \leq C(\Delta t + N^{-2}L^{2}).
\end{align*}
On similar lines we can show the following error bound for left singular component $Z_L$ in the inner region $[-1,-1+\tau] \times (0,T]$ 
\begin{align*}
|(Z_L-z_l)(x_i,t_n)| &\leq C(N^{-2}L^{2} \bar{B}_i + \Delta t),~for~0 \leq i \leq N/4,~n \Delta t \leq T \\
& \leq C(\Delta t + N^{-2}L^{2}).
\end{align*}
\end{proof}
Combining the error estimates for the regular and the singular components, $Y(x_i,t_n)$ and $Z(x_i,t_n)$, respectively, we can obtain the error estimate for the discrete solution $U(x_i,t_n)$ of the problem (\ref{3.2})-(\ref{3.2b}) at each mesh points $(x_i,t_n) \in \overline{Q}_{\tau}^{N,M}$.
\begin{theorem}
	\label{thm3.1}
	Let $u(x_i,t_n)$ be the exact solution of the problem (\ref{1.1})-(\ref{1.3}) and $U(x_i,t_n)$ be the discrete solution of the problem (\ref{3.2})-(\ref{3.2b}) at each mesh points $(x_i,t_n) \in \overline{Q}_{\tau}^{N,M}$. Then, for $ N \geq N_0 $ with the assumption (\ref{3.7}), we have
	\begin{align*}
	\|(U-u)\|_{\overline{Q}_{\tau}^{N,M}} \leq 
	C(\Delta t + N^{-2}L^{2} + N^{-2}).
	\end{align*} 
\end{theorem}

\section{Numerical Examples and results}
In this section, numerical experiments are conducted on two examples to show the efficiency and applicability of the proposed scheme. The numerical results demonstrate the high accuracy and convergence rate of the proposed finite difference scheme as compared to the upwind scheme on uniform mesh as well as upwind scheme on Shishkin mesh.
\begin{problem}
	Consider the following singularly perturbed parabolic IBVP :-
	\begin{equation} \label{5.1}
	\begin{cases}
	
	\left( \varepsilon \frac{\partial^2{u}}{\partial{x^2}}- 2(2x-1) \frac{\partial{u}}{\partial{x}}-\frac{\partial{u}}{\partial{t}}-4u\right)(x,t) = 0, ~ \text{for} ~ (x,t)\in Q = (0,1) \times (0,1],\\
	u(x,0)=1,~ \text{for} ~ x \in [0,1] ,\\
	u(-1,t)=1,~ u(1,t)=1,~ \text{for} ~t\in (0,1].
	
	\end{cases}
	\end{equation}
\end{problem}
\begin{problem}
	Consider the following singularly perturbed parabolic IBVP :-
	\begin{equation} \label{5.2}
	\begin{cases}
	
	\left( \varepsilon \frac{\partial^2{u}}{\partial{x^2}}- x^p \frac{\partial{u}}{\partial{x}}-\frac{\partial{u}}{\partial{t}}-u\right)(x,t) = 1, ~ \text{for} ~ (x,t)\in Q = (-1,1) \times (0,1],\\
	u(x,0)=1,~ \text{for} ~ x \in [-1,1] ,\\
	u(-1,t)=1,~ u(1,t)=1,~ \text{for} ~t\in (0,1].
	
	\end{cases}
	\end{equation}
\end{problem}

Since the analytical solutions of the considered problems are not known, the double mesh principle is used to estimate the maximum pointwise error as follows:
\begin{align*}
E^{N,M}_\varepsilon = \| {U}^{N,M}(x_i,t_n)- {U}^{2N,2M}(x_i,t_n)\|_{\overline{Q}^{N,M}_{\tau}},
\end{align*}
The corresponding order of convergence $q_{\varepsilon}^{N,M}$ is computed as
\begin{align*}
q^{N,M}_\varepsilon = \frac{\ln\left( {{E^{N,M}_\varepsilon}/E^{2N,2M}_\varepsilon}\right)}{\ln{2}}.
\end{align*}
Also, the $ \varepsilon $-uniform maximum point-wise error $E^{N,M}$ is computed as
\begin{align*}
E^{N,M} = \max\limits_{\varepsilon}E^{N,M}_\varepsilon,
\end{align*}
and the corresponding $ \varepsilon $-uniform order of convergence $q^{N,M}$ is given by
\begin{align*}
q^{N,M} = \frac{\ln{\left( {E^{N,M}}/E^{2N,2M}\right)}}{\ln{2}}.
\end{align*}
The maximum pointwise error $E_\varepsilon^{N,M}$ and the order of convergence $q_\varepsilon^{N,M}$ are computed in Tables 1-6 for Problems 1 and 2, corresponding to different values of $\varepsilon$ and $N$. The numerical solutions, errors and loglog plots are given in Figures 1-3 for Problems 1 and 2. \newline

Tables 1 and 3 display the results computed using upwind scheme on uniform mesh and piecewise uniform Shishkin mesh for Problems 1 and 2, respectively. From Tables 1 and 3, we can see that the numerical method on the uniform mesh is not $\varepsilon$-uniformly convergent. It is clearly evident that the upwind scheme on Shishkin mesh is $\varepsilon$-uniformly convergent with almost first-order accuracy.\newline
Tables 2 and 4 display the results computed using the hybrid scheme on the generalized Shishkin mesh  for Problems 1 and 2, respectively. The results clearly showcase the $\varepsilon$-uniform convergence of the proposed numerical scheme. Table 4 does not clearly reflect the actual theoretical order of convergence of the proposed hybrid scheme (\ref{3.2}) in space as proved in Theorem \ref{thm3.1} due to first order accuracy in time direction. To justify the spatial order of convergence we conduct the numerical experiments by taking $M=N^2$ and display the maximum pointwise errors $E_\varepsilon^{N,M}$ and the corresponding order of convergence $q_\varepsilon^{N,M}$ in Table 5.\newline
In Table 6, the $\varepsilon$-uniform maximum pointwise error $E^{N,M}$ and the corresponding order of convergence $q^{N,M}$ are tabulated to show that the proposed numerical scheme works equally well for different values of $p$.\newline
In Figure 1, the surface plots corresponding to Problems 1 and 2 are displayed which clearly shows that the solutions of the problems exhibit twin boundary layers. The higher efficiency of the proposed hybrid scheme compared to the upwind scheme is showcased in Figures 2 and 3 for Problems 1 and 2, respectively. \newline

\section{Conclusions}
A singularly perturbed parabolic convection-diffusion problem with multiple interior turning point exhibiting twin exponential boundary layers is examined. Analytical aspects of the problem are studied via obtaining theoretical bounds on the solution of the problem and its derivatives. A higher order numerical method is constructed comprising of implicit Euler method for the temporal discretization on uniform mesh and a hybrid numerical scheme on the generalized Shishkin mesh for the spatial discretization. The proposed method is $\varepsilon$-uniformly convergent of order one in the time variable and almost order two in the space variable. Further, numerical experiments are conducted to investigate the performance of the proposed scheme. The numerical results tabulated and the plots displayed clearly demonstrate the higher accuracy of the proposed hybrid scheme on the generalized Shishkin as compared to the simple upwind scheme on the uniform mesh as well as on the standard Shishkin mesh. The results obtained are in good agreement with the theoretical results.

\section*{Acknowledgements}
The first author wish to acknowledge UGC Non-NET Fellowships for financial support vide Ref.No. Sch./139/Non-NET/Ext-152/2018-19/67.

	\begin{table}[h!]
		\begin{center}
	    \caption{The maximum pointwise errors $ E^{N,M}_\varepsilon $ and the corresponding order of convergence $q^{N,M}_\varepsilon $ for the Problem 1 using the upwind scheme on uniform mesh and piecewise uniform Shishkin mesh.}
    	\resizebox{\textwidth}{!}{\begin{tabular}{ c  c  c  c  c  c  c  c  c }
    	\hline
		\textbf{$ \varepsilon \downarrow $} & \textbf{Schemes} & \textbf{$ N=32 $} & \textbf{$ N=64 $} & \textbf{$ N=128 $} & \textbf{$ N=256 $} & \textbf{$ N=512 $} & \textbf{$ N=1024 $} & \textbf{$ N=2048 $}\\
		\hline
		{$2^{-6}$} & \textbf{Upwind scheme on} & 8.65176e-02 & 8.05002e-02 & 5.36364e-02 & 3.35920e-02 & 1.91246e-02 & 1.02613e-02 & 5.32357e-03\\ 
	  & \textbf{uniform mesh} & 1.04000e-01 & 5.85781e-01 & 6.75095e-01 & 8.12684e-01 & 8.98225e-01 & 9.46743e-01 \\
		\hline
		{} & \textbf{Upwind scheme on} & 5.07853e-02 & 3.54176e-02 & 2.28512e-02 & 1.39745e-02 & 8.18788e-03 & 4.66192e-03 & 2.59945e-03\\ 
	&   \textbf{Shishkin mesh} & 5.19943e-01 & 6.32201e-01 & 7.09466e-01 & 7.71239e-01 & 8.12566e-01 & 8.42717e-01 \\
		\hline
		{$2^{-8}$} & \textbf{Upwind scheme on} & 4.73788e-02 & 7.02291e-02 & 8.70753e-02 & 8.14823e-02 & 5.43134e-02 & 3.40536e-02 & 1.94008e-02 \\
	&  \textbf{uniform mesh} & -5.67827e-01 & -3.10194e-01 & 9.57777e-02 & 5.85178e-01 & 6.73501e-01 & 8.11692e-01\\ 
		\hline
		{} & \textbf{Upwind scheme on} & 5.19597e-02 & 3.60848e-02 & 2.32711e-02 & 1.42267e-02 & 8.33626e-03 & 4.74664e-03 & 2.64683e-03 \\ 
    &  \textbf{Shishkin mesh} & 5.26004e-01 & 6.32849e-01 & 7.09940e-01 & 7.71129e-01 & 8.12493e-01 & 8.42643e-01 \\
		\hline
		{$2^{-10}$} & \textbf{Upwind scheme on} & 1.65023e-02 & 2.73950e-02 & 4.60476e-02 & 6.99109e-02 & 8.72143e-02 & 8.17259e-02 & 5.44819e-02 \\
    &  \textbf{uniform mesh} & -7.31245e-01 & -7.49216e-01 & -6.02391e-01 & -3.19047e-01 & 9.37705e-02 & 5.85017e-01\\
		\hline
		{} & \textbf{Upwind scheme on} & 5.22313e-02 & 3.62430e-02 & 2.33723e-02 & 1.42888e-02 & 8.37309e-03 & 4.76773e-03 & 2.65862e-03 \\ 
	&  \textbf{Shishkin mesh} & 5.27215e-01 & 6.32901e-01 & 7.09918e-01 & 7.71050e-01 & 8.12458e-01 & 8.42626e-01\\
		\hline
		{$2^{-12}$} & \textbf{Upwind scheme on} & 6.15711e-03 & 8.53035e-03 & 1.47539e-02 & 2.66066e-02 & 4.57306e-02 & 6.98340e-02 & 8.72492e-02\\
	&  \textbf{uniform mesh} & -4.70351e-01 & -7.90420e-01 & -8.50687e-01 & -7.81377e-01 & -6.10769e-01 & -3.21212e-01\\
		\hline
		{} & \textbf{Upwind scheme on} & 5.22969e-02 & 3.62812e-02 & 2.33966e-02 & 1.43038e-02 & 8.38210e-03 & 4.77295e-03 & 2.66155e-03\\ 
	&  \textbf{Shishkin mesh} & 5.27504e-01 & 6.32920e-01 & 7.09900e-01 & 7.71017e-01 & 8.12432e-01 & 8.42611e-01 \\
		\hline
		{$2^{-14}$} & \textbf{Upwind scheme on} & 3.37025e-03 & 3.07191e-03 & 4.33583e-03 & 7.65960e-03 & 1.43380e-02 & 2.64147e-02 & 4.56525e-02\\
	&  \textbf{uniform mesh} & 1.33717e-01 & -4.97170e-01 & -8.20961e-01 & -9.04502e-01 & -8.81498e-01 & -7.89355e-01\\
		\hline
		{} & \textbf{Upwind scheme on} & 5.23131e-02 & 3.62906e-02 & 2.34026e-02 & 1.43075e-02 & 8.38431e-03 & 4.77422e-03 & 2.66228e-03\\ 
	&  \textbf{Shishkin mesh} & 5.27577e-01 & 6.32926e-01 & 7.09897e-01 & 7.71010e-01 & 8.12426e-01 & 8.42605e-01 \\
		\hline
		{$2^{-16}$} & \textbf{Upwind scheme on} & 2.66032e-03 & 1.65591e-03 & 1.53370e-03 & 2.18569e-03 & 3.90294e-03 & 7.44770e-03 & 1.42354e-02\\
	&  \textbf{uniform mesh} & 6.83973e-01 & 1.10609e-01 & -5.11070e-01 & -8.36474e-01 & -9.32235e-01 & -9.34619e-01\\
		\hline
		{} & \textbf{Upwind scheme on} & 5.23172e-02 & 3.62930e-02 & 2.34041e-02 & 1.43084e-02 & 8.38485e-03 & 4.77454e-03 & 2.66246e-03 \\ 
	&  \textbf{Shishkin mesh} & 5.27595e-01 & 6.32928e-01 & 7.09896e-01 & 7.71009e-01 & 8.12424e-01 & 8.42604e-01 \\
		\hline
		{$2^{-18}$} & \textbf{Upwind scheme on} & 2.48200e-03 & 1.29861e-03 & 8.20163e-04 & 7.66210e-04 & 1.09730e-03 & 1.97007e-03 & 3.79623e-03\\
	&  \textbf{uniform mesh} & 9.34530e-01 & 6.62990e-01 & 9.81693e-02 & -5.18146e-01 & -8.44291e-01 & -9.46316e-01\\
		\hline
		{} & \textbf{Upwind scheme on} & 5.23182e-02 & 3.62936e-02 & 2.34045e-02 & 1.43087e-02 & 8.38499e-03 & 4.77462e-03 & 2.66250e-03\\ 
	&  \textbf{Shishkin mesh} & 5.27599e-01 & 6.32928e-01 & 7.09896e-01 & 7.71008e-01 & 8.12424e-01 & 8.42604e-01\\
		\hline
		{$2^{-20}$} & \textbf{Upwind scheme on} & 2.43737e-03 & 1.20908e-03 & 6.40955e-04 & 4.08070e-04 & 3.82936e-04 & 5.49766e-04 & 9.89729e-04\\
	& \textbf{uniform mesh} & 1.01141e+00 & 9.15617e-01 & 6.51405e-01 & 9.17148e-02 & -5.21716e-01 & -8.48215e-01\\
		\hline
		{} & \textbf{Upwind scheme on} & 5.23184e-02 & 3.62937e-02 & 2.34046e-02 & 1.43087e-02 & 8.38502e-03 & 4.77464e-03 & 2.66251e-03\\ 
	&  \textbf{Shishkin mesh} & 5.27600e-01 & 6.32928e-01 & 7.09896e-01 & 7.71008e-01 & 8.12424e-01 & 8.42604e-01\\
		\hline
		{$2^{-22}$} & \textbf{Upwind scheme on} & 2.42620e-03 & 1.18669e-03 & 5.96101e-04 & 3.18329e-04 & 2.03524e-04 & 1.91424e-04 & 2.75162e-04\\
	&  \textbf{uniform mesh} & 1.03176e+00 & 9.93309e-01 & 9.05037e-01 & 6.45319e-01 & 8.84269e-02 & -5.23509e-01 \\
		\hline
		{} & \textbf{Upwind scheme on} & 5.23185e-02 & 3.62937e-02 & 2.34046e-02 & 1.43087e-02 & 8.38503e-03 & 4.77464e-03 & 2.66251e-03\\ 
	&  \textbf{Shishkin mesh} & 5.27601e-01 & 6.32928e-01 & 7.09896e-01 & 7.71008e-01 & 8.12424e-01 & 8.42604e-01 \\
		\hline
		{$2^{-24}$} & \textbf{Upwind scheme on} & 2.42341e-03 & 1.18109e-03 & 5.84884e-04 & 2.95881e-04 & 1.58620e-04 & 1.01633e-04 & 9.57011e-05\\
	&  \textbf{uniform mesh} & 1.03693e+00 & 1.01389e+00 & 9.83134e-01 & 8.99443e-01 & 6.42199e-01 & 8.67676e-02\\
		\hline
		{} & \textbf{Upwind scheme on} & 5.23185e-02 & 3.62938e-02 & 2.34046e-02 & 1.43087e-02 & 8.38504e-03 & 4.77464e-03 & 2.66251e-03\\ 
	&  \textbf{Shishkin mesh} & 5.27601e-01 & 6.32928e-01 & 7.09896e-01 & 7.71008e-01 & 8.12424e-01 & 8.42604e-01 \\
	     \hline
	$ \mathbf{E^{N,M}}   $ & \textbf{Upwind scheme on} & \textbf{8.65176e-02} & \textbf{8.05002e-02} & \textbf{8.70753e-02} & \textbf{8.14823e-02} & \textbf{8.72143e-02} & \textbf{8.17259e-02} & \textbf{8.72492e-02}\\
		$ \mathbf{p^{N,M}}$ & \textbf{uniform mesh} & \textbf{1.04000e-01} & \textbf{-1.13270e-01} & \textbf{9.57777e-02} & \textbf{-9.80785e-02} & \textbf{9.37705e-02} & \textbf{-9.43478e-02}\\
		\hline
	$ \mathbf{E^{N,M}}$ & \textbf{Upwind scheme on} & \textbf{5.23185e-02} & \textbf{3.62938e-02} & \textbf{2.34046e-02} & \textbf{1.43087e-02} & \textbf{8.38504e-03} & \textbf{4.77464e-03} & \textbf{2.66251e-03}\\
		$ \mathbf{p^{N,M}}$ & \textbf{Shishkin mesh} & \textbf{5.27601e-01} & \textbf{6.32928e-01} & \textbf{7.09896e-01} & \textbf{7.71008e-01} & \textbf{8.12424e-01} & \textbf{8.42604e-01}\\
		\hline
		
	\end{tabular}}
\end{center}
\end{table}

\newpage

\begin{table}[h!]
\begin{center}
\caption{The maximum pointwise errors $ E^{N,M}_\varepsilon $ and the corresponding order of convergence $q^{N,M}_\varepsilon $ for the Problem 1 using the hybrid scheme on the generalized Shishkin mesh.}
\resizebox{\textwidth}{!}{\begin{tabular}{ c  c  c  c  c  c  c  c   }	
		\hline
		\textbf{$ \varepsilon \downarrow $} & \textbf{$ N=32 $} & \textbf{$ N=64 $} & \textbf{$ N=128 $} & \textbf{$ N=256 $} & \textbf{$ N=512 $} & \textbf{$ N=1024 $} & \textbf{$ N=2048 $}\\
		\hline 
		{$2^{-6}$} & 4.76939e-02 & 1.55289e-02 & 4.92103e-03 & 1.54015e-03 & 4.64216e-04 & 1.31311e-04 & 3.58076e-05 \\ 
	& 1.61885e+00 & 1.65792e+00 & 1.67589e+00 & 1.73020e+00 & 1.82181e+00 & 1.87465e+00 \\
		\hline
		{$2^{-8}$} & 5.07576e-02 & 1.63691e-02 & 5.18695e-03 & 1.62227e-03 & 4.90864e-04 & 1.39166e-04 & 3.58076e-05\\
	& 1.63265e+00 & 1.65802e+00 & 1.67687e+00 & 1.72462e+00 & 1.81852e+00 & 1.95847e+00\\ 
		\hline
		{$2^{-10}$} & 5.15288e-02 & 1.65802e-02 & 5.25362e-03 & 1.64287e-03 & 4.97547e-04 & 1.41138e-04 & 3.60069e-05\\
	& 1.63592e+00 & 1.65808e+00 & 1.67710e+00 & 1.72331e+00 & 1.81773e+00 & 1.97076e+00\\
		\hline
		{$2^{-12}$} & 5.17220e-02 & 1.66331e-02 & 5.27030e-03 & 1.64802e-03 & 4.99219e-04 & 1.41631e-04 & 3.61501e-05\\
	& 1.63672e+00 & 1.65810e+00 & 1.67715e+00 & 1.72299e+00 & 1.81753e+00 & 1.97007e+00\\
		\hline
		{$2^{-14}$} & 5.17703e-02 & 1.66463e-02 & 5.27447e-03 & 1.64931e-03 & 4.99637e-04 & 1.41755e-04 & 3.61859e-05\\
	& 1.63692e+00 & 1.65810e+00 & 1.67716e+00 & 1.72291e+00 & 1.81748e+00 & 1.96990e+00\\
		\hline
		{$2^{-16}$} & 5.17823e-02 & 1.66496e-02 & 5.27551e-03 & 1.64963e-03 & 4.99741e-04 & 1.41786e-04 & 3.61948e-05\\
	& 1.63697e+00 & 1.65810e+00 & 1.67717e+00 & 1.72289e+00 & 1.81747e+00 & 1.96985e+00\\
		\hline
		{$2^{-18}$} & 5.17854e-02 & 1.66504e-02 & 5.27578e-03 & 1.64971e-03 & 4.99767e-04 & 1.41793e-04 & 3.61971e-05\\
	& 1.63699e+00 & 1.65810e+00 & 1.67717e+00 & 1.72289e+00 & 1.81747e+00 & 1.96984e+00\\
		\hline
		{$2^{-20}$} & 5.17861e-02 & 1.66506e-02 & 5.27584e-03 & 1.64973e-03 & 4.99774e-04 & 1.41795e-04 & 3.61976e-05\\
	& 1.63699e+00 & 1.65810e+00 & 1.67717e+00 & 1.72289e+00 & 1.81747e+00 & 1.96984e+00\\
		\hline
		{$2^{-22}$} & 5.17863e-02 & 1.66507e-02 & 5.27586e-03 & 1.64974e-03 & 4.99775e-04 & 1.41796e-04 & 3.61978e-05\\
	& 1.63699e+00 & 1.65810e+00 & 1.67717e+00 & 1.72288e+00 & 1.81747e+00 & 1.96984e+00 \\
		\hline
		{$2^{-24}$} & 5.17864e-02 & 1.66507e-02 & 5.27586e-03 & 1.64974e-03 & 4.99776e-04 & 1.41796e-04 & 3.61978e-05\\
	& 1.63699e+00 & 1.65810e+00 & 1.67717e+00 & 1.72288e+00 & 1.81747e+00 & 1.96984e+00\\
		\hline
	$ \mathbf{E^{N,M}}   $ & \textbf{5.17864e-02} & \textbf{1.66507e-02} & \textbf{5.27586e-03} & \textbf{1.64974e-03} & \textbf{4.99776e-04} & \textbf{1.41796e-04} & \textbf{3.61978e-05}\\
		$ \mathbf{p^{N,M}}$ & \textbf{1.63699e+00} & \textbf{1.65810e+00} & \textbf{1.67717e+00} & \textbf{1.72288e+00} & \textbf{1.81747e+00} & \textbf{1.96984e+00}\\
		\hline
		
	\end{tabular}}
	\end{center}
\end{table}

\newpage

\begin{table}[h!]
\begin{center}
\caption{The maximum pointwise errors $ E^{N,M}_\varepsilon $ and the corresponding order of convergence $q^{N,M}_\varepsilon $ for the Problem 2 using the upwind scheme on uniform mesh and piecewise uniform Shishkin mesh.}
	\resizebox{\textwidth}{!}{\begin{tabular}{ c  c  c  c  c  c  c  c  c }		
		\hline
		\textbf{$ \varepsilon \downarrow $} & \textbf{Schemes} & \textbf{$ N=32 $} & \textbf{$ N=64 $} & \textbf{$ N=128 $} & \textbf{$ N=256 $} & \textbf{$ N=512 $} & \textbf{$ N=1024 $} & \textbf{$ N=2048 $}\\
		\hline
		{$2^{-6}$} & \textbf{Upwind scheme on}  & 1.13453e-01 & 1.02884e-01 & 6.76358e-02 & 4.22961e-02 & 2.40523e-02 & 1.29099e-02 & 6.69412e-03\\ 
	& \textbf{uniform mesh} & 1.41078e-01 & 6.05157e-01 & 6.77263e-01 & 8.14348e-01 & 8.97704e-01 & 9.47506e-01\\
		\hline
		{} & \textbf{Upwind scheme on} & 4.21726e-02 & 2.90406e-02 & 1.76665e-02 & 1.01097e-02 & 5.63613e-03 & 3.10684e-03 & 1.69969e-03\\ 
	& \textbf{Shishkin mesh} & 5.38233e-01 & 7.17059e-01 & 8.05281e-01 & 8.42956e-01 & 8.59256e-01 & 8.70173e-01\\
		\hline
		{$2^{-8}$} & \textbf{Upwind scheme on} & 6.66332e-02 & 9.38616e-02 & 1.13505e-01 & 1.05022e-01 & 6.96679e-02 & 4.36559e-02 & 2.48621e-02\\
	& \textbf{uniform mesh} & -4.94295e-01 & -2.74142e-01 & 1.12051e-01 & 5.92133e-01 & 6.74316e-01 & 8.12227e-01\\ 
		\hline
		{} & \textbf{Upwind scheme on} & 4.28790e-02 & 2.95701e-02 & 1.83352e-02 & 1.07210e-02 & 6.02918e-03 & 3.31096e-03 & 1.79677e-03\\ 
	& \textbf{Shishkin mesh} & 5.36133e-01 & 6.89523e-01 & 7.74177e-01 & 8.30403e-01 & 8.64711e-01 & 8.81843e-01\\
		\hline
		{$2^{-10}$} & \textbf{Upwind scheme on} & 2.50205e-02 & 3.78181e-02 & 6.09623e-02 & 9.09846e-02 & 1.12695e-01 & 1.05283e-01 & 7.00973e-02\\
	& \textbf{uniform mesh} & -5.95962e-01 & -6.88841e-01 & -5.77706e-01 & -3.08733e-01 & 9.81517e-02 & 5.86844e-01\\
		\hline
		{} & \textbf{Upwind scheme on} & 4.26028e-02 & 2.90409e-02 & 1.79139e-02 & 1.05342e-02 & 6.02106e-03 & 3.36637e-03 & 1.84426e-03\\ 
	& \textbf{Shishkin mesh} & 5.52865e-01 & 6.97007e-01 & 7.65993e-01 & 8.06995e-01 & 8.38823e-01 & 8.68152e-01\\
		\hline
		{$2^{-12}$} & \textbf{Upwind scheme on} & 1.07982e-02 & 1.26842e-02 & 2.00424e-02 & 3.49330e-02 & 5.93213e-02 & 9.01773e-02 & 1.12459e-01\\
	& \textbf{uniform mesh} & -2.32242e-01 & -6.60028e-01 & -8.01536e-01 & -7.63958e-01 & -6.04215e-01 & -3.18557e-01\\
		\hline
		{} & \textbf{Upwind scheme on} & 4.24955e-02 & 2.88346e-02 & 1.77151e-02 & 1.03505e-02 & 5.89787e-03 & 3.30847e-03 & 1.83398e-03\\ 
	& \textbf{Shishkin mesh} & 5.59509e-01 & 7.02825e-01 & 7.75282e-01 & 8.11430e-01 & 8.34030e-01 & 8.51187e-01\\
		\hline
		{$2^{-14}$} & \textbf{Upwind scheme on} & 6.94865e-03 & 5.38129e-03 & 6.36470e-03 & 1.03072e-02 & 1.87310e-02 & 3.41866e-02 & 5.89004e-02\\
	& \textbf{uniform mesh} & 3.68781e-01 & -2.42141e-01 & -6.95489e-01 & -8.61775e-01 & -8.68003e-01 & -7.84846e-01\\
		\hline
		{} & \textbf{Upwind scheme on} & 4.24662e-02 & 2.87775e-02 & 1.76555e-02 & 1.02871e-02 & 5.83986e-03 & 3.26114e-03 & 1.80219e-03\\ 
	& \textbf{Shishkin mesh} & 5.61370e-01 & 7.04826e-01 & 7.79276e-01 & 8.16836e-01 & 8.40559e-01 & 8.55626e-01\\
		\hline
		{$2^{-16}$} & \textbf{Upwind scheme on} & 5.96687e-03 & 3.48485e-03 & 2.68220e-03 & 3.18574e-03 & 5.22551e-03 & 9.70323e-03 & 1.83997e-02\\
	& \textbf{uniform mesh} & 7.75877e-01 & 3.77678e-01 & -2.48212e-01 & -7.13941e-01 & -8.92894e-01 & -9.23144e-01\\
		\hline
		{} & \textbf{Upwind scheme on} & 4.24587e-02 & 2.87629e-02 & 1.76399e-02 & 1.02705e-02 & 5.82268e-03 & 3.24498e-03 & 1.78757e-03\\ 
	& \textbf{Shishkin mesh} & 5.61849e-01 & 7.05367e-01 & 7.80336e-01 & 8.18751e-01 & 8.43476e-01 & 8.60212e-01\\
		\hline
		{$2^{-18}$} & \textbf{Upwind scheme on} & 5.72019e-03 & 3.00620e-03 & 1.74425e-03 & 1.33858e-03 & 1.59345e-03 & 2.63078e-03 & 4.93893e-03\\
	& \textbf{uniform mesh} & 9.28122e-01 & 7.85331e-01 & 3.81905e-01 & -2.51453e-01 & -7.23334e-01 & -9.08708e-01\\
		\hline
		{} & \textbf{Upwind scheme on} & 4.24568e-02 & 2.87592e-02 & 1.76360e-02 & 1.02663e-02 & 5.81820e-03 & 3.24057e-03 & 1.78324e-03\\ 
	& \textbf{Shishkin mesh} & 5.61969e-01 & 7.05505e-01 & 7.80610e-01 & 8.19266e-01 & 8.44324e-01 & 8.61746e-01\\
		\hline
		{$2^{-20}$} & \textbf{Upwind scheme on} & 5.65845e-03 & 2.88625e-03 & 1.50867e-03 & 8.72506e-04 & 6.68615e-04 & 7.96841e-04 & 1.31991e-03\\
	& \textbf{uniform mesh} & 9.71208e-01 & 9.35920e-01 & 7.90040e-01 & 3.83990e-01 & -2.53117e-01 & -7.28071e-01\\
		\hline
		{} & \textbf{Upwind scheme on} & 4.24563e-02 & 2.87583e-02 & 1.76350e-02 & 1.02652e-02 & 5.81706e-03 & 3.23945e-03 & 1.78211e-03\\ 
	& \textbf{Shishkin mesh} & 5.62000e-01 & 7.05539e-01 & 7.80680e-01 & 8.19397e-01 & 8.44543e-01 & 8.62159e-01\\
		\hline
		{$2^{-22}$} & \textbf{Upwind scheme on} & 5.64300e-03 & 2.85625e-03 & 1.44971e-03 & 7.55718e-04 & 4.36339e-04 & 3.34133e-04 & 3.98445e-04\\
	& \textbf{uniform mesh} & 9.82342e-01 & 9.78362e-01 & 9.39840e-01 & 7.92397e-01 & 3.85028e-01 & -2.53958e-01\\
		\hline
		{} & \textbf{Upwind scheme on} & 4.24562e-02 & 2.87581e-02 & 1.76347e-02 & 1.02649e-02 & 5.81678e-03 & 3.23916e-03 & 1.78183e-03\\ 
	& \textbf{Shishkin mesh} & 5.62007e-01 & 7.05548e-01 & 7.80697e-01 & 8.19430e-01 & 8.44599e-01 & 8.62264e-01\\
		\hline
		{$2^{-24}$} & \textbf{Upwind scheme on} & 5.63914e-03 & 2.84875e-03 & 1.43496e-03 & 7.26504e-04 & 3.78204e-04 & 2.18190e-04 & 1.67022e-04\\
	& \textbf{uniform mesh} & 9.85149e-01 & 9.89317e-01 & 9.81968e-01 & 9.41807e-01 & 7.93577e-01 & 3.85547e-01\\
		\hline
		{} & \textbf{Upwind scheme on} & 4.24561e-02 & 2.87580e-02 & 1.76347e-02 & 1.02649e-02 & 5.81671e-03 & 3.23909e-03 & 1.78176e-03\\ 
	& \textbf{Shishkin mesh} & 5.62009e-01 & 7.05550e-01 & 7.80701e-01 & 8.19438e-01 & 8.44613e-01 & 8.62290e-01\\
		\hline
	$ \mathbf{E^{N,M}}   $ & \textbf{Upwind scheme on} & \textbf{1.13453e-01} & \textbf{1.02884e-01} & \textbf{1.13505e-01} & \textbf{1.05022e-01} & \textbf{1.12695e-01} & \textbf{1.05283e-01} & \textbf{1.12459e-01}\\
		$ \mathbf{p^{N,M}}   $ & \textbf{uniform mesh} & \textbf{1.41078e-01} & \textbf{-1.41734e-01} & \textbf{1.12051e-01} & \textbf{-1.01728e-01} & \textbf{9.81517e-02} & \textbf{-9.51197e-02} \\
		\hline
		$ \mathbf{E^{N,M}}$ & \textbf{Upwind scheme on} & \textbf{4.28790e-02} & \textbf{2.95701e-02} & \textbf{1.83352e-02} & \textbf{1.07210e-02} & \textbf{6.02918e-03} & \textbf{3.36637e-03}& \textbf{1.84426e-03}\\
		$ \mathbf{p^{N,M}}$ & \textbf{Shishkin mesh} & \textbf{5.36133e-01} & \textbf{6.89523e-01} & \textbf{7.74177e-01} & \textbf{8.30403e-01} & \textbf{8.40768e-01} & \textbf{8.68152e-01}\\
		\hline
	\end{tabular}}
	\end{center}
\end{table}

\newpage

\begin{table}[h!]
\begin{center}
\caption{The maximum pointwise errors $ E^{N,M}_\varepsilon $ and the corresponding order of convergence $q^{N,M}_\varepsilon $ for the Problem 2 using the hybrid scheme on the generalized Shishkin mesh with $p=3$.}
	\resizebox{\textwidth}{!}{\begin{tabular}{ c  c  c  c  c  c  c  c   }
		\textbf{$ \varepsilon \downarrow $} & \textbf{$ N=32 $} & \textbf{$ N=64 $} & \textbf{$ N=128 $} & \textbf{$ N=256 $} & \textbf{$ N=512 $} & \textbf{$ N=1024 $} & \textbf{$ N=2048 $}\\
		\hline      
		{$2^{-6}$} & 1.31600e-02 & 4.59802e-03 & 1.43004e-03 & 7.16764e-04 & 3.58819e-04 & 1.79519e-04 & 8.97870e-05\\ 
	& 1.51708e+00 & 1.68495e+00 & 9.96489e-01 & 9.98242e-01 & 9.99120e-01 & 9.99560e-01\\
		\hline
		{$2^{-8}$} & 1.58487e-02 & 4.76817e-03 & 1.46684e-03 & 7.63095e-04 & 3.58819e-04 & 1.79519e-04 & 8.97869e-05\\
	& 1.73286e+00 & 1.70072e+00 & 9.42778e-01 & 1.08860e+00 & 9.99120e-01 & 9.99560e-01\\ 
		\hline
		{$2^{-10}$} & 1.64697e-02 & 4.99240e-03 & 1.44004e-03 & 7.21264e-04 & 3.60532e-04 & 1.81544e-04 & 8.97869e-05\\
	& 1.72201e+00 & 1.79362e+00 & 9.97510e-01 & 1.00040e+00 & 9.89812e-01 & 1.01574e+00\\
		\hline
		{$2^{-12}$} & 1.66249e-02 & 5.04828e-03 & 1.44653e-03 & 7.18025e-04 & 3.59389e-04 & 1.79750e-04 & 8.98677e-05\\
	& 1.71949e+00 & 1.80320e+00 & 1.01049e+00 & 9.98490e-01 & 9.99551e-01 & 1.00012e+00\\
		\hline
		{$2^{-14}$} & 1.66637e-02 & 5.06223e-03 & 1.45146e-03 & 7.17533e-04 & 3.59017e-04 & 1.79589e-04 & 8.98166e-05\\
	& 1.71887e+00 & 1.80227e+00 & 1.01638e+00 & 9.98991e-01 & 9.99352e-01 & 9.99647e-01\\
		\hline
		{$2^{-16}$} & 1.66734e-02 & 5.06572e-03 & 1.45269e-03 & 7.17539e-04 & 3.59021e-04 & 1.79568e-04 & 8.97982e-05\\
	& 1.71871e+00 & 1.80204e+00 & 1.01760e+00 & 9.98988e-01 & 9.99539e-01 & 9.99772e-01\\
		\hline
		{$2^{-18}$} & 1.66759e-02 & 5.06659e-03 & 1.45300e-03 & 7.17541e-04 & 3.59022e-04 & 1.79569e-04 & 8.97984e-05\\
	& 1.71867e+00 & 1.80198e+00 & 1.01790e+00 & 9.98988e-01 & 9.99539e-01 & 9.99772e-01\\
		\hline
		{$2^{-20}$} & 1.66765e-02 & 5.06681e-03 & 1.45308e-03 & 7.17541e-04 & 3.59023e-04 & 1.79569e-04 & 8.97985e-05\\
	& 1.71866e+00 & 1.80197e+00 & 1.01798e+00 & 9.98987e-01 & 9.99538e-01 & 9.99772e-01\\
		\hline
		{$2^{-22}$} & 1.66766e-02 & 5.06686e-03 & 1.45310e-03 & 7.17541e-04 & 3.59023e-04 & 1.79569e-04 & 8.97985e-05\\
	& 1.71866e+00 & 1.80196e+00 & 1.01800e+00 & 9.98987e-01 & 9.99538e-01 & 9.99772e-01\\
		\hline
		{$2^{-24}$} & 1.66767e-02 & 5.06688e-03 & 1.45310e-03 & 7.17541e-04 & 3.59023e-04 & 1.79569e-04 & 8.97985e-05\\
	& 1.71866e+00 & 1.80196e+00 & 1.01800e+00 & 9.98987e-01 & 9.99538e-01 & 9.99772e-01\\
		\hline
	$ \mathbf{E^{N,M}}   $ & \textbf{1.66767e-02} & \textbf{5.06688e-03} & \textbf{1.46684e-03} & \textbf{7.63095e-04} & \textbf{3.60532e-04} & \textbf{1.81544e-04} & \textbf{8.98677e-05}\\
		$ \mathbf{p^{N,M}}$ & \textbf{1.71866e+00} & \textbf{1.78838e+00} & \textbf{9.42778e-01} & \textbf{1.08173e+00} & \textbf{9.89812e-01} & \textbf{1.01444e+00}\\
		\hline
		
	\end{tabular}}
	\end{center}
	\end{table}

\newpage

\begin{table}[h!]
\begin{center}
\caption{The maximum pointwise errors $ E^{N,M}_\varepsilon $ and the corresponding order of convergence $q^{N,M}_\varepsilon $ for the Problem 2 using the hybrid scheme on the generalized Shishkin mesh with $M=N^2$ and $p=3$.}
	\resizebox{\textwidth}{!}{\begin{tabular}{ c  c  c  c  c  c    }
		\hline
		\textbf{$ \varepsilon \downarrow $} & \textbf{$ N=32 $} & \textbf{$ N=64 $} & \textbf{$ N=128 $} & \textbf{$ N=256 $} & \textbf{$ N=512 $}\\
		\hline
		{$2^{-6}$} & 1.61692e-02 & 6.74563e-03 & 1.89181e-03 & 6.31041e-04 & 2.03613e-04 \\ 
	    & 1.26122e+00 & 1.83419e+00 & 1.58396e+00 & 1.63190e+00\\
		\hline
		{$2^{-8}$} & 1.89880e-02 & 6.50248e-03 & 2.17098e-03 & 7.70847e-04 & 2.33877e-04\\
	    & 1.54602e+00 & 1.58264e+00 & 1.49383e+00 & 1.72070e+00\\ 
		\hline
		{$2^{-10}$} & 1.96402e-02 & 6.74166e-03 & 2.25106e-03 & 7.47893e-04 & 2.41407e-04 \\
	    & 1.54264e+00 & 1.58250e+00 & 1.58970e+00 & 1.63136e+00 \\
		\hline
		{$2^{-12}$} & 1.98031e-02 & 6.80120e-03 &     2.27317e-03 & 7.54483e-04 & 2.43306e-04\\
	    & 1.54187e+00 & 1.58109e+00 & 1.59114e+00 & 1.63272e+00\\
		\hline
		{$2^{-14}$} & 1.98439e-02 & 6.81607e-03 & 2.27873e-03 & 7.56505e-04 & 2.43837e-04\\
	 & 1.54168e+00 & 1.58071e+00 & 1.59081e+00 & 1.63343e+00\\
		\hline
		{$2^{-16}$} & 1.98540e-02 & 6.81978e-03 & 2.28013e-03 & 7.57016e-04 & 2.44028e-04\\
	& 1.54164e+00 & 1.58061e+00 & 1.59072e+00 & 1.63328e+00 \\
		\hline
		{$2^{-18}$} & 1.98566e-02 & 6.82071e-03 & 2.28047e-03 & 7.57145e-04 & 2.44076e-04 \\
	& 1.54162e+00 & 1.58059e+00 & 1.59069e+00 & 1.63324e+00 \\
		\hline
		{$2^{-20}$} & 1.98572e-02 & 6.82094e-03 & 2.28056e-03 & 7.57177e-04 & 2.44088e-04 \\
	& 1.54162e+00 & 1.58058e+00 & 1.59069e+00 & 1.63323e+00\\
		\hline
		{$2^{-22}$} & 1.98574e-02 & 6.82100e-03 & 2.28058e-03 & 7.57185e-04 & 2.44091e-04 \\
	& 1.54162e+00 & 1.58058e+00 & 1.59069e+00 & 1.63323e+00 \\
		\hline
		{$2^{-24}$} & 1.98574e-02 & 6.82102e-03 & 2.28059e-03 & 7.57187e-04 & 2.44092e-04 \\
	& 1.54162e+00 & 1.58058e+00 & 1.59068e+00 & 1.63323e+00\\
		\hline
	$ \mathbf{E^{N,M}}   $ & \textbf{1.98574e-02} & \textbf{6.82102e-03} & \textbf{2.28059e-03} & \textbf{7.70847e-04} & \textbf{2.44092e-04} \\
		$ \mathbf{p^{N,M}}$ & \textbf{1.54162e+00} & \textbf{1.58058e+00} & \textbf{1.56489e+00} & \textbf{1.65902e+00} \\
		\hline		
	\end{tabular}}
	\end{center}
	\end{table}

\newpage

\begin{table}[h!]
\begin{center}
\caption{The $\varepsilon$-uniform maximum pointwise errors $ E^{N,M} $ and the corresponding order of convergence $q^{N,M} $ for the Problem 2 using the  hybrid scheme on the generalized Shishkin mesh with $M=N^2$ corresponding to different values of $p$.}
	\resizebox{\textwidth}{!}{\begin{tabular}{ c  c  c  c  c  c  c }
		\hline
		\textbf{$ p \downarrow $}& \textbf{$ N=32 $} & \textbf{$ N=64 $} & \textbf{$ N=128 $} & \textbf{$ N=256 $} & \textbf{$ N=512 $} \\
		\hline
		\textbf{1} & 1.98574e-02 & 6.82102e-03 & 2.28059e-03 & 7.71391e-04 & 2.44092e-04\\ 
		 & 1.54162e+00 & 1.58058e+00 & 1.56387e+00 & 1.66004e+00\\
		\midrule
		\textbf{5} & 1.98574e-02 & 6.82101e-03 & 2.28059e-03 & 7.81153e-04 & 2.44092e-04 \\
		& 1.54162e+00 & 1.58058e+00 & 1.54573e+00 & 1.67818e+00  \\
		\midrule
		{\textbf{7}} & 2.29690e-02 & 6.82101e-03 & 2.28059e-03 & 8.06812e-04 & 2.44092e-04\\
		& 1.75163e+00 & 1.58058e+00 & 1.49910e+00 & 1.72481e+00\\
		\midrule
		 {\textbf{9}} & 3.19970e-02 & 6.82101e-03 & 2.28059e-03 & 8.52480e-04 & 2.44091e-04\\
		& 2.22988e+00 & 1.58058e+00 & 1.41967e+00 & 1.80424e+00 \\ 
		\hline
		
	\end{tabular}}
\end{center}
\end{table}


\begin{thebibliography}{10}

\bibitem{becher} {\sc S. Becher and H.-G. Roos}, {\em Richardson extrapolation for a singularly perturbed turning point problem
with exponential boundary layers}, J. Comput. Appl. Math. 290 (2015), pp. 334-351.

\bibitem{farrell} {\sc  P.~A. Farrell, A.~F. Hegarty, J.~J.~H.~Miller, E.~O'Riordan and G.~I.~Shishkin}, {\em Robust
Computational Techniques for Boundary Layers}, Vol. 16 of Applied Mathematics (Boca Raton). Chapman \& Hall/CRC, Boca Raton, FL, 2000. 

\bibitem{geng2013} {\sc F. Geng and S. P. Qian}, {\em Reproducing kernel method for singularly perturbed turning point problems
having twin boundary layers}, Appl. Math. Lett. 26 (2013), pp. 998-1004.

\bibitem{kadalbajoo2011} {\sc M.~K. Kadalbajoo, P. Arora and V. Gupta}, {\em Collocation method using artificial viscosity for solving stiff singularly perturbed turning point problem having twin boundary layers}, Comput. Math. Appl. 61 (2011), pp. 1595-1607.

\bibitem{kadalbajoo2010} {\sc M.~K. Kadalbajoo and V. Gupta}, {\em A parameter-uniform B-spline collocation method for solving singularly perturbed turning point problem having twin boundary layers}, Int. J. Comput. Math. 87 (2010), pp. 3218-3235. 

\bibitem{kadalbajoo2001} {\sc M.~K. Kadalbajoo and K.~C. Patidar}, {\em Variable mesh spline approximation method for solving singularly perturbed turning point problems having boundary layer(s)}, Comput. Math. Appl. 42 (2001),
pp. 1439-1453.

\bibitem{kumar2019} {\sc D. Kumar}, {\em A parameter-uniform method for singularly perturbed turning point problems exhibiting
	interior or twin boundary layers}, Int. J. Comput. Math. 96 (2019), pp. 865-882.

\bibitem{MR0241822}
{\sc O.~A. Ladyzenskaja, V.~A. Solonnikov, and N.~N. Ural'ceva}, {\em Linear and quasilinear equations of parabolic type}, Translated from the Russian by S. Smith. Translations of Mathematical Monographs, Vol. 23, American Mathematical Society, Providence, R.I., 1968.

\bibitem{miller} {\sc J.~J.~H. Miller,  E. O'Riordan and G.~I.~Shishkin}, {\em Fitted Numerical Methods for Singular
Perturbation Problems: error estimates
in the maximum norm for linear problems in one and two dimensions}, World Scientific Publishing Co., Inc., River Edge, NJ, 1996.
  
\bibitem{natesan2003} {\sc S. Natesan, J. Jayakumar and J. Vigo-Aguiar}, {\em Parameter-uniform numerical method for singularly
 	perturbed turning point problems exhibiting boundary layers}, J. Comput. Appl. Math. 158 (2003), pp. 121-134.
  
\bibitem{Potter}
{\sc M.~H. Protter and H.~F. Weinberger}, {\em Maximum principles in differential equations}, Springer Science \& Business Media, 2012.

\bibitem{roos}  {\sc H.-G. Roos, M. Stynes and L. Tobiska}, {\em Numerical Methods for Singularly Perturbed
Differential Equations: convection-diffusion and 
flow problems}, Vol. 24 of Springer Series in Computational Mathematics. Springer-Verlag, Berlin, 1996.

\bibitem{vulanovic} {\sc R. Vulanovi\`c and P.A. Farrell}, {\em Analysis of multiple turning point problems}, Rad. Mat. 8 (1992/96), pp. 105-113.

\end{thebibliography}
\end{document}